\documentclass[a4paper, intlimits, reqno]{amsart}

\usepackage[english]{babel}
\usepackage[T1]{fontenc}
\usepackage[utf8]{inputenc}

\usepackage{amsmath}
\usepackage{amssymb}
\usepackage{MnSymbol}
\usepackage{amsthm}
\allowdisplaybreaks
\usepackage{amsfonts}
\usepackage{mathrsfs} 
\usepackage{mathtools}
\usepackage{nicefrac}
\usepackage{enumitem}
\usepackage{multicol}
\usepackage{url}
\usepackage{dsfont}
\usepackage[numbers]{natbib}
\usepackage{prettyref}

\newrefformat{def}{Definition \ref{#1}}
\newrefformat{rem}{Remark \ref{#1}}
\newrefformat{lem}{Lemma \ref{#1}}
\newrefformat{sect}{Section \ref{#1}}
\newrefformat{prop}{Proposition \ref{#1}}
\newrefformat{thm}{Theorem \ref{#1}}
\newrefformat{cor}{Corollary \ref{#1}}
\newrefformat{ex}{Example \ref{#1}}

\swapnumbers
\newtheoremstyle{dotless}{}{}{\itshape}{}{\bfseries}{}{}{}
\theoremstyle{dotless}
\theoremstyle{plain}
\newtheorem{thm}{Theorem}[section]
\newtheorem{lem}[thm]{Lemma}
\newtheorem{prop}[thm]{Proposition}
\newtheorem{cor}[thm]{Corollary}
\theoremstyle{definition}
\newtheorem{defn}[thm]{Definition}
\newtheorem{rem}[thm]{Remark}

\newcommand{\N} {\mathbb{N}}
\newcommand{\R} {\mathbb{R}}
\newcommand{\C} {\mathbb{C}}
\newcommand{\K} {\mathbb{K}}
\newcommand{\D} {\mathbb{D}}

\newcommand{\oacx} {\overline{\operatorname{acx}}}

\providecommand{\differential}{\mathrm{d}}
\renewcommand{\d}{\differential}

\begin{document}

\title[Holomorphic]{Vector-valued holomorphic functions in several variables}
\author[K.~Kruse]{Karsten Kruse}
\address{Hamburg University of Technology \\ Institute of Mathematics \\
Am Schwarzenberg-Campus~3 \\
21073 Hamburg \\
Germany}
\email{karsten.kruse@tuhh.de}



\date{\today}
\begin{abstract}
In the present paper we give some explicit proofs for folklore theorems on holomorphic functions 
in several variables with values in a locally complete locally convex Hausdorff space $E$ over $\C$. 
Most of the literature on vector-valued holomorphic functions is 
either devoted to the case of one variable or to infinitely many variables 
whereas the case of (finitely many) several variables is only touched or 
is subject to stronger restrictions on the completeness of $E$ like sequential completeness. 
The main tool we use is Cauchy's integral formula for derivatives for an $E$-valued holomorphic function 
in several variables which we derive via Pettis-integration. 
This allows us to generalise the known integral formula, 
where usually a Riemann-integral is used,
from sequentially complete $E$ to locally complete $E$. Among the classical theorems 
for holomorphic functions in several variables with values in 
a locally complete space $E$ we prove are the identity theorem, Liouville's theorem, 
Riemann's removable singularities theorem and the 
density of the polynomials in the $E$-valued polydisc algebra. 
\end{abstract}

\maketitle

\section{Introduction}

This is not a survey article but a comprehensive treatment of vector-valued holomorphic
functions in several variables, i.e.\ holomorphic functions $f\colon \Omega\to E$ from 
an open set $\Omega\subset\C^{d}$ to a complex locally convex Hausdorff space $E$. 
We give complete proofs in the case when $E$ is locally complete; some of these proofs are new in this context,
some are only touched upon in the literature with reference to the case of one
variable.

There is a lot of work available on $\C$-valued holomorphic functions in several variables, 
like the books by Gunning and Rossi \cite{gunning1965}, H\"ormander \cite{H3}, Jarnicki and Pflug \cite{pflug2008} 
and Krantz \cite{krantz2001}. 
But when it comes to vector-valued holomorphic functions,
then most of the work is either restricted to the case of one variable, i.e.\ $d=1$, 
or directly jumps to infinitely many variables, i.e.\ 
$\Omega$ is an open subset of a complex infinite dimensional locally convex Hausdorff space $F$. 
Holomorphy of vector-valued functions in infinitely many variables is discussed for instance by Mujica in \cite{mujica1985}, where 
$F$ and $E$ are Banach spaces, and by Dineen in \cite{dineen1981}, \cite{dineen1999} for general locally convex spaces $F$ and $E$. 
These references also contain results on finitely many variables ($F=\C^{d}$) but the emphasis is on infinitely many variables. 

Banach-valued holomorphic functions in one variable are handled and characterised by Dunford \cite[Theorem 76, p.\ 354]{Dunford1938} 
and more recently by Arendt and Nikolski \cite{Arendt2000}, \cite{Arendt2006}.
Holomorphic functions in one variable with values in a locally convex Hausdorff space $E$ are considered in 
\cite[Satz 10.11, p.\ 241]{Kaballo} by Kaballo if $E$ is quasi-complete, in \cite{Grothendieck1953} by Grothendieck 
if $E$ has the convex compactness property (cf.\ \cite[16.7.2 Theorem, p.\ 362-363]{Jarchow}), in \cite{Bogdanowicz1969} by Bogdanowicz 
if $E$ is sequentially complete, in \cite{grosse-erdmann1992}, \cite{grosse-erdmann2004} by Grosse-Erdmann 
if $E$ is locally complete and several equivalent conditions describing holomorphy are given. 
In particular, in all these cases holomorphy coincides with \emph{weak holomorphy} 
which means that $f\colon \Omega\to E$ is holomorphic if and only if the $\C$-valued functions $e'\circ f$ are holomorphic 
for each $e'\in E'$ where $E'$ is the dual space of $E$. Further, the interesting question is treated under which conditions 
one can replace $E'$ by a separating subspace $G\subset E'$ and still can conclude holomorphy 
from the holomorphy of $e'\circ f$ for each $e'\in G$.
More generally, the extension problem for $E$-valued holomorphic functions 
which have weakly holomorphic extensions is studied, in one variable 
by Grosse-Erdmann in \cite{grosse-erdmann2004}, in several variables by Bonet, Frerick and Jord\'a in \cite{B/F/J}, \cite{F/J} 
and Vitali's and Harnack's type results are derived in \cite{jorda2007} if $E$ is locally complete.
Further results on vector-valued holomorpic functions in several variables may be found in \cite{bochnak1971} by Bochnak and Siciak 
where $E$ is sequentially complete and in a survey by Barletta and Dragomir \cite{barletta2009_1}, 
extended in \cite{barletta2009_2}, but here $E$ is often restricted to have the convex compactness property 
or even to be a Fr\'echet space. 

The main purpose of the present paper is to derive some equivalent characterisations of holomorphic functions in several variables 
with values in a locally complete space $E$ (see \prettyref{cor:hol.in.O}, 
\prettyref{thm:holom.equiv}, \prettyref{cor:hol_infinitely_compl_diff}) 
with explicit proofs avoiding the usual \emph{`like in the case of one variable'} 
(see e.g.\ the four-line \cite[Section 4.1, p.\ 409]{grosse-erdmann2004}). 
Of course, the short reference to the case of one variable is often due 
to constraints, like page limits or the perception as folklore since it is known to everyone from the field how to transfer the results 
from one variable to several variables but never written down, not least because of the low chance to get it published.
This is the reason why we wrote this down so that we have a reference with explicit proofs, 
not more, not less. Anyway, our main tool to obtain the equivalent characterisations 
of holomorphic functions in several variables with values 
in a locally complete space $E$ is Cauchy's integral formula for derivatives which we obtain 
via Pettis-integration (\prettyref{thm:cauchy.int.formula}). 
To the best of our knowledge Cauchy's integral formula for derivatives for holomorphic functions 
with values in a locally complete space $E$ is not contained in the literature. 
Usually, Riemann-integration is used instead of Pettis-integration and $E$ has to be sequentially complete or the derivatives have 
to be considered in the completion of $E$. On the way to our main \prettyref{thm:holom.equiv} 
we derive Fubini's theorem (\prettyref{thm:fubini}) 
and Leibniz' rule for differentiation under the integral sign (\prettyref{lem:compl.diff.under.int}) 
for holomorphic functions with values in a 
locally complete space. We use our main theorem to prove some classical theorems like the identity theorem (\prettyref{thm:identity}), 
Liouville's theorem (\prettyref{thm:liouville}), Riemann's removable singularities theorem (\prettyref{thm:riemann}) and the 
density of the polynomials in the $E$-valued polydisc algebra (\prettyref{cor:disc-algebra}).

\section{Notation and Preliminaries}

We equip the spaces $\R^{d}$ and $\C^{d}$, $d\in\N$, with the usual Euclidean norm $|\cdot|$. 
Moreover, we denote by $\mathbb{B}_{r}(x):=\{w\in\R^{d}\;|\;|w-x|<r\}$ the ball around $x\in\R^{d}$ 
with radius $r>0$ and use the same notation when $\R^{d}$ is replaced by $\C^{d}$. 
Furthermore, for a subset $M$ of a topological space $X$ we denote by $\overline{M}$ the closure of $M$ in $X$.
For a subset $M$ of a topological vector space $X$, we write $\oacx(M)$ 
for the closure of the \emph{absolutely convex hull} $\operatorname{acx}(M)$ of $M$ in $X$.

By $E$ we always denote a non-trivial locally convex Hausdorff space (\emph{lcHs}) over the field 
$\K=\R$ or $\C$ equipped with a directed fundamental system of 
seminorms $(p_{\alpha})_{\alpha\in \mathfrak{A}}$. 
If $E=\K$, then we set $(p_{\alpha})_{\alpha\in\mathfrak{A}}:=\{|\cdot|\}$. 
Further, we write $\widehat{E}$ for the completion of $E$ and for a disk $D\subset E$, i.e.\ a bounded, absolutely convex set, 
we write $E_{D}:=\bigcup_{n\in\N}nD$ which becomes a normed vector space if it is equipped with the
gauge functional of $D$ as a norm (see \cite[p.\ 151]{Jarchow}). The space $E$ is called \emph{locally 
complete} if $E_{D}$ is a Banach space for every closed disk $D\subset E$ (see \cite[10.2.1 Proposition, p.\ 197]{Jarchow}). 
In particular, every sequentially complete space is locally complete and this implication is strict. 
Further, we recall the following definitions from \cite[p.\ 259]{J.Voigt}
and \cite[9-2-8 Definition, p.\ 134]{Wilansky}. A locally convex Hausdorff space is said 
to have the \emph{[metric] convex compactness property ([metric] ccp)} if 
the closure of the absolutely convex hull of every [metrisable] compact set is compact. 
Equivalently this definition can be phrased with the convex hull instead of the absolutely convex hull. 
Every locally convex Hausdorff space with ccp has metric ccp, every quasi-complete locally convex 
Hausdorff space has ccp, every sequentially complete locally convex Hausdorff space 
has metric ccp and every locally convex Hausdorff space with metric ccp is locally complete and
all these implications are strict (see \cite[p.\ 1512-1513]{kruse2020} and the references therein).
For more details on the theory of locally convex spaces see \cite{F/W/Buch}, \cite{Jarchow} or \cite{meisevogt1997}.

For $k\in\N_{0,\infty}:=\N_{0}\cup\{\infty\}$ we denote by $\mathcal{C}^{k}(\Omega,E)$ the space of $k$-times continuously partially differentiable 
functions on an open set $\Omega\subset\R^{d}$ with values in a locally convex Hausdorff space $E$. We say that a function $f\colon\Omega\to E$ 
is weakly $\mathcal{C}^{k}$ if $e'\circ f\in \mathcal{C}^{k}(\Omega):=\mathcal{C}^{k}(\Omega,\K)$ for each $e'\in E'$.
By $L(F,E)$ we denote the space of continuous linear operators from $F$ to $E$ where $F$ and $E$ are locally convex Hausdorff spaces. 
If $E=\K$, we just write $F':=L(F,\K)$ for the dual space.
We write $L_{t}(F,E)$ for the space $L(F,E)$ equipped with the locally convex topology of uniform convergence on 
compact subsets of $F$ if $t=c$, on the absolutely convex, compact subsets of $F$ if $t=\kappa$ and on the bounded subsets of $F$ if $t=b$.
The so-called \emph{$\varepsilon$-product of Schwartz} is defined by 
\[
F\varepsilon E:=L_{e}(F_{\kappa}',E)
\]
where $L(F_{\kappa}',E)$ is equipped with the topology of uniform convergence on the equicontinuous subsets of $F'$ 
(see e.g.\ \cite[Chap.\ I, \S1, D\'{e}finition, p.\ 18]{Sch1}). 
For more information on the theory of $\varepsilon$-products see \cite{Jarchow} and \cite{Kaballo}.

\section{Notions of differentiability}

\begin{defn}[{(weakly, separately, G\^{a}teaux-) differentiable, holomorphic}]\label{def:holom}
Let $E$ be an lcHs over $\K$, let $\Omega\subset\K^{d}$ be open and $f\colon \Omega \to E$.
\begin{enumerate}
\item [a)] $f$ is called \emph{differentiable} (on $\Omega$) if for every $z\in\Omega$ there is a $\K$-linear map 
$df(z):=d_{\K}f(z)\colon\K^{d}\to\widehat{E}$ such that
\[
\lim_{\substack{w\to z\\w\in\Omega,w\neq z}}\frac{f(w)-f(z)-df(z)[w-z]}{|w-z|}=0\quad\text{in}\;\widehat{E}
\]
and the map $df(\cdot)[v]\colon\Omega\to \widehat{E}$ is continuous for every $v\in\K^{d}$. 
\item[b)] $f$ is called the \emph{G\^{a}teaux-differentiable} (on $\Omega$) if 
\[
 Df(z)[v]:=D_{\K}f(z)[v]:=\lim_{\substack{h\to 0\\h\in\K,h \neq 0}}\frac{f(z+h v)-f(z)}{h}
 \quad\text{exists in}\; \widehat{E}
\]
for every $z\in\Omega$ and $v\in\K^{d}$.
\item [c)] If $v=e_{j}$ is the $j$th unit vector for $1\leq j\leq d$ and $z\in\Omega$, 
we write 
\[
(\partial_{\K}^{e_{j}})^{E}f(z):=(\partial_{z_{j}})^{E}f(z):=D_{\K}f(z)[e_{j}]
\]
if $D_{\K}f(z)[e_{j}]$ exists in $E$. 
Especially, we use $f'(z):=(\partial_{\K}^{e_{1}})^{E}f(z)$ if $d=1$.
\item [d)] For $z=(z_{1},\ldots,z_{d})\in\Omega$ we define the continuous function
\[
 \pi_{z,j}\colon \K\to \K^{d},\;\pi_{z,j}(w):=(z_{1},\ldots,z_{j-1},w,z_{j+1},\ldots,z_{d}).
\]
$f$ is called \emph{separately differentiable} (on $\Omega$)
if $f$ is a differentiable function in each variable, 
i.e.\ $f\circ\pi_{z,j}\colon \pi_{z,j}^{-1}(\Omega)\to E$ is differentiable for every $z\in\Omega$ and 
$1\leq j\leq d$.
\item [e)] $f$ is called \emph{weakly (separately, G\^{a}teaux-) differentiable} (on $\Omega$) if 
$e'\circ f\colon \Omega \to\K$ is (separately, G\^{a}teaux-) differentiable for every $e'\in E'$. 
\item [f)] If $\K=\C$, we say \emph{holomorphic} or complex differentiable instead of differentiable on the open set $\Omega$ and, 
if $\K=\R$, we sometimes say \emph{real differentiable}.
\end{enumerate}
\end{defn}

\begin{rem}\label{rem:properties.of.holom.map}
Let $E$ be an lcHs over $\K$, $\Omega\subset\K^{d}$ open and $f\colon\Omega\to E$.  
\begin{enumerate}
\item [a)] If $f$ is differentiable, then $df\colon\Omega\times\K^{d}\to \widehat{E}$ is continuous.
\item [b)] If $f$ is differentiable, then $f\colon\Omega\to E$ is continuous. 
\item [c)] If $f$ is differentiable, then $f$ is G\^{a}teaux- and separately differentiable and 
\[
df(z)[v]=Df(z)[v]=\sum_{j=1}^{d}\left(\partial_{\K}^{e_{j}}\right)^{\widehat{E}}f(z)v_{j},
\quad z\in\Omega,\,v=(v_{1},\ldots,v_{d})\in\K^{d}.
\]
\item [d)] If $f$ is (separately, G\^{a}teaux-) differentiable, then $f$ is weakly (separately, G\^{a}teaux-) differentiable.
\item [e)] If $(\partial_{\K}^{e_{j}})^{\widehat{E}}f(z)\in E$ for some $1\leq j\leq d$ and $z\in\Omega$, then 
\[
(\partial_{\K}^{e_{j}})^{\widehat{E}}f(z)=(\partial_{\K}^{e_{j}})^{E}f(z).
\]
\end{enumerate}
\end{rem}
\begin{proof}
a) First, we remark that $df(z)\colon\K^{d}\to \widehat{E}$ is continuous 
for every $z\in\Omega$ since $df(z)$ is linear and $\K^{d}$ a finite dimensional normed space.
Let $(z,v)\in\Omega\times\K^{d}$, $\varepsilon>0$ and $\alpha\in\widehat{\mathfrak{A}}$ where 
$(\widehat{E},(p_{\alpha})_{\alpha\in\widehat{\mathfrak{A}}})$ is the completion of $E$. 
For every $(w,x)\in\Omega\times\K^{d}$ we estimate
\begin{flalign*}
&\quad p_{\alpha}(df(w)[x]-df(z)[v])\\
&\leq p_{\alpha}(df(w)[x-v])+p_{\alpha}(df(w)[v]-df(z)[v])\\
&\leq\sqrt{d}\sup_{1\leq j\leq d}p_{\alpha}(df(w)[e_{j}])|x-v|+p_{\alpha}(df(w)[v]-df(z)[v]).
\end{flalign*}
Since $df(\cdot)[v]\colon\Omega\to\widehat{E}$ is continuous, there is $\delta=\delta_{\alpha,z,v}>0$ such that 
for all $w\in\Omega$ with $|w-z|<\delta$ we have 
\[
p_{\alpha}(df(w)[v]-df(z)[v])<\varepsilon/2.
\]
As $\Omega$ is open, there is $\delta_{0}>0$ such that $K_{z}:=\overline{\mathbb{B}_{\delta_{0}}(z)}\subset\Omega$. 
From the compactness of $K_{z}$ and the continuity of $df(\cdot)[e_{j}]\colon\Omega\to\widehat{E}$ 
for every $1\leq j\leq d$ we deduce that 
\[
C_{j,z}:=\sup_{w\in K_{z}}p_{\alpha}(df(w)[e_{j}])<\infty.
\]
Thus we obtain for every $(w,x)\in\Omega\times\K^{d}$ with
\[
 |(w,x)-(z,v)|<\min\Bigl(\delta,\delta_{0},\frac{\varepsilon}{2(1+\sqrt{d}\sup_{1\leq j\leq d}C_{j,z})}\Bigr)
\]
that 
\[
 p_{\alpha}(df(w)[x]-df(z)[v])\leq \sqrt{d}\sup_{1\leq j\leq d}C_{j,z}|x-v|+(\varepsilon/2)<(\varepsilon/2)+(\varepsilon/2)=\varepsilon.
\]

b) Let $z\in\Omega$, $\varepsilon>0$ and $\alpha\in\widehat{\mathfrak{A}}$. 
Then there is $\delta>0$ such that for all $w\in\K^{d}$ with $0<|w-z|<\delta$ 
we have
\[
p_{\alpha}\Bigl(\frac{f(w)-f(z)}{|w-z|}\Bigr)-p_{\alpha}\Bigl(\frac{df(z)[w-z]}{|w-z|}\Bigr)
\leq p_{\alpha}\Bigl(\frac{f(w)-f(z)-df(z)[w-z]}{|w-z|}\Bigr)<1.
\]
It follows from the continuity of $df(z)\colon\K^{d}\to \widehat{E}$ that there is $C>0$ such that
\[
p_{\alpha}\Bigl(\frac{f(w)-f(z)}{|w-z|}\Bigr)<1+p_{\alpha}\Bigl(\frac{df(z)[w-z]}{|w-z|}\Bigr)
\leq 1+C\frac{|w-z|}{|w-z|}=1+C.
\]
Thus we have for all $w\in\K^{d}$ with $|w-z|<\min(\varepsilon/(1+C),\delta)$ that
\[
p_{\alpha}(f(w)-f(z))\leq (1+C)|w-z|<\varepsilon,
\]
implying $f\in\mathcal{C}^{0}(\Omega,\widehat{E})$. Since $f(\Omega)\subset E$, we derive that $f\colon\Omega\to E$ is continuous.

c) Let $z\in\Omega$, $v\in\K^{d}$, $v\neq 0$, $\alpha\in\widehat{\mathfrak{A}}$ and $h\in\K$, $h\neq 0$, 
such that $z+hv\in\Omega$. We observe that
\[
 p_{\alpha}\Bigl(\frac{f(z+h v)-f(z)}{h}-df(z)[v]\Bigr)
= |v|p_{\alpha}\Bigl(\frac{f(z+h v)-f(z)-df(z)[hv]}{|h v|}\Bigr),
\]
which yields the G\^{a}teaux-differentiability of $f$ and $Df(z)[v]=df(z)[v]$ because $f$ is differentiable. 
Due to the linearity of $df(z)$ we obtain for $v=(v_{1},\ldots,v_{d})\in\K^{d}$ 
\[
df(z)[v]=\sum_{j=1}^{d}df(z)[e_{j}]v_{j}=\sum_{j=1}^{d}Df(z)[e_{j}]v_{j}
=\sum_{j=1}^{d}\left(\partial_{\K}^{e_{j}}\right)^{\widehat{E}}f(z)v_{j}.
\]
Finally, let $1\leq j\leq d$ and $z_{0}\in\pi_{z,j}^{-1}(\Omega)\subset\K$. 
Clearly, $v_{0}\mapsto df(\pi_{z,j}(z_{0}))[e_{j}]\cdot v_{0}$ is a linear map from $\K$ to $\widehat{E}$ 
and $df(\pi_{z,j}(\cdot))[e_{j}]\cdot v_{0}\colon \pi_{z,j}^{-1}(\Omega)\to \widehat{E}$ is continuous for every $v_{0}\in\C$
by the differentiability of $f$ and the continuity of $\pi_{z,j}$.
We set $\widetilde{z}:=\pi_{z,j}(z_{0})$ and get for $w_{0}\in\pi_{z,j}^{-1}(\Omega)$, $w_{0}\neq z_{0}$, that
\begin{flalign*}
 &\quad p_{\alpha}\Bigl(\frac{(f\circ\pi_{z,j})(w_{0})-(f\circ\pi_{z,j})(z_{0})-df(\pi_{z,j}(z_{0}))[e_{j}]
 \cdot(w_{0}-z_{0})}{|w_{0}-z_{0}|}\Bigr)\\
 &=p_{\alpha}\Bigl(\frac{f(\widetilde{z}+(w_{0}-z_{0})e_{j})-f(\widetilde{z})}{w_{0}-z_{0}}-df(\widetilde{z})[e_{j}]\Bigr)\\
 &=p_{\alpha}\Bigl(\frac{f(\widetilde{z}+(w_{0}-z_{0})e_{j})-f(\widetilde{z})}{w_{0}-z_{0}}-Df(\widetilde{z})[e_{j}]\Bigr).
\end{flalign*}
Letting $w_{0}\to z_{0}$, we derive that $f$ is separately differentiable and 
\[
 d(f\circ\pi_{z,j})(z_{0})[v_{0}]=df(\pi_{z,j}(z_{0}))[e_{j}]\cdot v_{0},\quad v_{0}\in\K.
\]

d) We just have to observe that $(\widehat{E})'=E'$ by \cite[3.4.4 Corollary, p.\ 63]{Jarchow} and get for every $e'\in E'$ 
\[
 d(e'\circ f)= e'\circ df,\quad D(e'\circ f)=e'\circ Df,\quad d(e'\circ f\circ \pi_{z,j})=e'\circ d(f\circ \pi_{z,j})
\]
for differentiable, G\^{a}teaux-differentiable and separately differentiable $f$, respectively. 

e) Follows directly from Definition \ref{def:holom} c) and the fact that $\widehat{E}$ is Hausdorff by 
\cite[3.3.2 Theorem, p.\ 60]{Jarchow}.
\end{proof}

We denote by
\[
\phi\colon\C^{d}\to\R^{2d},\;
\phi(\operatorname{Re}z_{1}+i\operatorname{Im}z_{1},\ldots,\operatorname{Re}z_{d}+i\operatorname{Im}z_{d})
:=(\operatorname{Re}z_{1},\operatorname{Im}z_{1},\ldots,\operatorname{Re}z_{d},\operatorname{Im}z_{d}),
\]
the isometric isomorphism between $\C^{d}$ and $\R^{2d}$ with respect 
to Euclidean norm on both sides and remark the following.

\begin{rem}\label{rem:compl-diff.is.real-diff}
Let $E$ be an lcHs over $\C$, $\Omega\subset\C^{d}$ open and $f\colon\Omega\to E$ 
holomorphic. Then $f\circ\phi^{-1}$ is real differentiable on $\phi(\Omega)$ and 
\begin{equation}\label{eq:compl-diff.is.real-diff.1}
d_{\R}(f\circ\phi^{-1})(x)[v]=d_{\C}f(\phi^{-1}(x))[\phi^{-1}(v)],
\quad x\in\phi(\Omega),\, v\in\R^{2d}.
\end{equation}
In particular, $f\circ\phi^{-1}\in\mathcal{C}^{1}(\phi(\Omega),\widehat{E})$.
\end{rem}
\begin{proof}
$E$ is also an lcHs over $\R$ and 
equation \eqref{eq:compl-diff.is.real-diff.1} follows from 
\begin{flalign*}
&\quad\frac{(f\circ\phi^{-1})(y)-(f\circ\phi^{-1})(x)-d_{\C}f(\phi^{-1}(x))[\phi^{-1}(x-y)]}{|x-y|}\\
&=\frac{f(\phi^{-1}(y))-f(\phi^{-1}(x))-d_{\C}f(\phi^{-1}(x))[\phi^{-1}(x)-\phi^{-1}(y)]}{|\phi^{-1}(x)-\phi^{-1}(y)|}
\end{flalign*}
for $x,y\in\phi(\Omega)$, $x\neq y$, and the holomorphy of $f$ on $\Omega$.
The $\R$-linearity of $d_{\R}(f\circ\phi^{-1})(x)$ for every $x\in\phi(\Omega)$ and
the continuity of $d_{\R}(f\circ\phi^{-1})(\cdot)[v]\colon\phi(\Omega)\to \widehat{E}$ 
for every $v\in\R^{2d}$ are a direct consequence of \eqref{eq:compl-diff.is.real-diff.1}.

Since $f$ is continuous on $\Omega$ by \prettyref{rem:properties.of.holom.map} b), 
the map $f\circ\phi^{-1}$ is continuous on $\phi(\Omega)$. 
Further, if $e_{j}$ is the $j$th unit vector in $\R^{2d}$, we obtain 
\[
(\partial^{e_{j}}_{\R})^{\widehat{E}}(f\circ\phi^{-1})(x)=d_{\R}(f\circ\phi^{-1})(x)[e_{j}]
=d_{\C}f(\phi^{-1}(x))[\phi^{-1}(e_{j})],\quad x\in\phi(\Omega).
\]
It follows that $f\circ\phi^{-1}$ is continuously partially (real) differentiable on $\phi(\Omega)$ 
because $d_{\C}f(\cdot)[\phi^{-1}(e_{j})]$ is continuous for every $1\leq j\leq 2d$.
\end{proof}

\section{Curve integrals via the Pettis-integral}

Our next goal is to derive Cauchy's integral formula (for derivatives) for a holomorphic function with values in a locally 
complete lcHs $E$. We use the notion of a Pettis-integral to define integration of a vector-valued function. 

\begin{defn}[Pettis-integral]\label{def:integral}
 	Let $\Omega\subset\R^{d}$, $E$ an lcHs, $(\Omega, \mathscr{L}(\Omega), \lambda)$ be the measure space of Lebesgue measurable sets 
 	and $\mathcal{L}^{1}(\Omega,\lambda)$ the space of $\K$-valued Lebesgue-integrable (equivalence classes of) functions on $\Omega$.
 	A function $f\colon \Omega \to E$ is called weakly measurable if the function 
 	$e'\circ  f  \colon X\to \K$, $ (e'\circ f )(x) := \langle e' , f(x) \rangle:=e'(f(x)),$ 
 	is Lebesgue measurable for all $e' \in E'$.  
 	A weakly measurable function is said to be weakly integrable 
 	if $e' \circ f \in \mathcal{L}^{1}(\Omega,\lambda)$. 
 	A function $f\colon \Omega\to E$ is called \emph{Pettis-integrable} on $\Lambda\in\mathscr{L}(\Omega)$ 
 	if it is weakly integrable on $\Lambda$ and
 	\[
 	\exists\; e_{\Lambda}(f) \in E \; \forall e' \in E': \langle e' , e_{\Lambda}(f) \rangle 
 	= \int\limits_{\Lambda} \langle e' , f(x) \rangle \d x. 
 	\]  
 	In this case $e_{\Lambda}(f)$ is unique due to $E$ being Hausdorff and we define the \emph{Pettis-integral} of $f$ on $\Lambda$ by 
 	\[
 	 \int\limits_{\Lambda} f(x) \d x:=e_{\Lambda}(f).
 	\]	
\end{defn}	 

A function $\gamma\colon [a,b]\to \C$ is called a \emph{$\mathcal{C}^{1}$-curve} (in $\C$) if $\gamma$ can be extended to 
a continuously differentiable function on an open set $X\subset\R$ with $[a,b]\subset X$. 
For a family $(\gamma_{k})_{1\leq k\leq d}$ of $\mathcal{C}^{1}$-curves $\gamma_{k}\colon [a,b]\to \C$ a function
\[
 \gamma\colon [a,b]^{d}\to\C^{d},\;\gamma(t_{1},\ldots,t_{d}):=(\gamma_{1}(t_{1}),\ldots,\gamma_{d}(t_{d})),
\]
is called a \emph{$\mathcal{C}^{1}$-curve} (in $\C^{d}$) and we set 
\[
l(\gamma):=\int_{[a,b]^{d}}\prod_{k=1}^{d}|\gamma_{k}'(t_{k})|\d t=
\prod_{k=1}^{d}\int_{a}^{b}|\gamma_{k}'(t_{k})|\d t_{k}
\]
which is the product of the \emph{lengths} of the curves $\gamma_{k}$.
We say that $\gamma$ is a \emph{$\mathcal{C}^{1}$-curve in $\Omega\subset\C^{d}$} 
if there are open sets $X_{k}\subset\R$ such that $[a,b]\subset X_{k}$ and $\gamma_{k}$ can be extended to 
a continuously differentiable function $\widetilde{\gamma}_{k}$ on $X_{k}$ for every $1\leq k\leq d$ and 
the so-defined extension $\widetilde{\gamma}:=(\widetilde{\gamma}_{k})_{k}$ of $\gamma$ 
on the open set $X:=\prod_{1\leq k\leq d}X_{k}\subset\R^{d}$ 
fulfils $\widetilde{\gamma}(X)\subset\Omega$.

Let $E$ be an lcHs over $\C$, $\Omega\subset\C^{d}$ and $\gamma\colon [a,b]^{d}\to \C^{d}$ be a $\mathcal{C}^{1}$-curve in $\Omega$. 
We define the (Pettis)-integral of a function $f\colon\Omega\to E$ along $\gamma$ by 
\[
 \int_{\gamma}f(z)\d z:=\int_{[a,b]^{d}}f(\gamma(t))\prod_{k=1}^{d}\gamma_{k}'(t_{k})\d t
\]
if the Pettis-integral on the right-hand side exists. If the integral exists, we call $f$ \emph{integrable along $\gamma$}.
Since $\gamma$ is a $\mathcal{C}^{1}$-curve in $\Omega$, 
there is some open set $X\subset\R^{d}$ 
such that $[a,b]^{d}\subset X$ and $\gamma$ can be extended to a $\mathcal{C}^{1}$-function $\widetilde{\gamma}$ on $X$ with 
$\widetilde{\gamma}(X)\subset\Omega$. If the extension of the factor of the integrand given by
\[
 g\colon X\to E,\; g(t):=f(\widetilde{\gamma}(t)),
\]
is a weakly $\mathcal{C}^{1}$ function on $X$, we call $f$ \emph{weakly $\gamma$-$\mathcal{C}^{1}$}. 

\begin{prop}\label{prop:path.integral}
Let $E$ be a locally complete lcHs over $\C$, $\Omega\subset\C^{d}$ open, 
$\gamma$ a $\mathcal{C}^{1}$-curve in $\Omega$ and $f\colon\Omega\to E$. 
\begin{enumerate}
\item [a)] If $f$ is weakly $\gamma$-$\mathcal{C}^{1}$, then $f$ is integrable along $\gamma$.
\item [b)] If $f\circ\phi^{-1}\colon \phi(\Omega) \to E$ is weakly $\mathcal{C}^{1}$, then $f$ is weakly 
$\gamma$-$\mathcal{C}^{1}$ in $\Omega$. 
\item [c)] If $f\colon\Omega\to E$ is holomorphic, then $f$ is weakly $\gamma$-$\mathcal{C}^{1}$.
\end{enumerate}
\end{prop}
\begin{proof}
a) As $f$ is weakly $\gamma$-$\mathcal{C}^{1}$, there is some open set $X\subset\R^{d}$ 
such that $[a,b]^{d}\subset X$ and $\gamma$ can be extended to a $\mathcal{C}^{1}$-function $\widetilde{\gamma}$ on $X$ with 
$\widetilde{\gamma}(X)\subset\Omega$ so that $f\circ\widetilde{\gamma}$ is weakly $\mathcal{C}^{1}$ on $X$. We observe that
\[
|I_{f}(e')|:=\bigl|\int_{[a,b]^{d}}\langle e',f(\gamma(t))\rangle\prod_{k=1}^{d}\gamma_{k}'(t_{k})\d t\bigr|
\leq l(\gamma)\sup_{x\in f(\gamma([a,b]^{d}))}|e'(x)|,\quad e'\in E'.
\]
The closure of the absolutely convex hull $\oacx{f(\gamma([a,b]^{d}))}=\oacx{f(\widetilde{\gamma}([a,b]^{d}))}$ 
of $f(\gamma([a,b]^{d}))$ is compact by \cite[Proposition 2, p.\ 354]{Bonet2002} 
since $f\circ\widetilde{\gamma}$ is weakly $\mathcal{C}^{1}$ on $X$. 
Hence it follows that $I_{f}\in (E_{\kappa}')'$ 
and we deduce from the Mackey-Arens theorem that there is $e(f\circ\gamma)\in E$ such that 
\[
 \langle e', e(f\circ\gamma)\rangle=I_{f}(e')
=\int_{[a,b]^{d}}\langle e',f(\gamma(t))\rangle\prod_{k=1}^{d}\gamma_{k}'(t_{k})\d t,\quad e'\in E',
\]
implying the integrability of $f$ along $\gamma$. 

b) Indeed, writing 
\[
e'\circ\bigl(f\circ\widetilde{\gamma}\bigr)=\bigl(e'\circ(f\circ\phi^{-1})\bigr)\circ(\phi\circ\widetilde{\gamma}),\quad e'\in E',
\]
for a $\mathcal{C}^{1}$-extension $\widetilde{\gamma}$ of $\gamma$ on $X$, 
we see that $f\circ\widetilde{\gamma}$ is weakly $\mathcal{C}^{1}$ 
on $X$ by the scalar version of the chain rule.

c) We just have to notice 
that $f$ is weakly holomorphic by \prettyref{rem:properties.of.holom.map} d), implying that 
$e'\circ(f\circ\phi^{-1})\in\mathcal{C}^{\infty}(\phi(\Omega))$ for all $e'\in E'$, 
which proves the claim by part b). 
\end{proof}

Next, we prove Fubini's theorem, which facilitates the computation of an integral along a curve. 
We recall the following lemma whose proof 
is similar to the one of \prettyref{prop:path.integral} a).

\begin{lem}[{\cite[4.7 Lemma, p.\ 14]{kruse2018_1}}]\label{lem:pettis.loc.complete}
Let $E$ be a locally complete lcHs, $\Omega\subset\R^{d}$ open and $f\colon\Omega\to E$.
If $f$ is weakly $\mathcal{C}^{1}$, then $f$ is Pettis-integrable (w.r.t.\ to the Lebesgue measure) 
on every compact subset of $K\subset\Omega$. 
\end{lem}

\begin{thm}[{Fubini's theorem}]\label{thm:fubini}
 Let $E$ be a locally complete lcHs, $\Omega\subset\R^{2}$ open, $[a,b]\times [c,d]\subset\Omega$ 
 and $f\colon\Omega\to E$ weakly $\mathcal{C}^{1}$. Then $f$ is Pettis-integrable on $[a,b]\times [c,d]$ and
 \[
  \int_{[a,b]\times[c,d]}f(x_{1},x_{2})\d (x_{1},x_{2})= \int_{[c,d]}\int_{[a,b]}f(x_{1},x_{2})\d x_{1}\d x_{2}
  =\int_{[a,b]}\int_{[c,d]}f(x_{1},x_{2})\d x_{2}\d x_{1}.
 \]
\end{thm}
\begin{proof}
The function $f$ is Pettis-integrable on $[a,b]\times [c,d]$ by \prettyref{lem:pettis.loc.complete}. 
Since $\Omega$ is open, there are $\widetilde{a}<a$, $b<\widetilde{b}$, $\widetilde{c}<c$ and $d<\widetilde{d}$ such that 
$[\widetilde{a},\widetilde{b}\,]\times[\widetilde{c},\widetilde{d}\,]\subset \Omega$. Further, we observe that
\[
 F\colon (\widetilde{c},\widetilde{d}\,)\to E,\;F(x_{2}):=\int_{[a,b]}f(x_{1},x_{2})\d x_{1},
\]
is well-defined by \prettyref{lem:pettis.loc.complete} since $f(\cdot,x_{2})$ 
is weakly $\mathcal{C}^{1}$ on $(\widetilde{a},\widetilde{b}\,)$ 
for every $x_{2}\in(\widetilde{c},\widetilde{d}\,)$.
We claim that $F$ is weakly $\mathcal{C}^{1}$ on $(\widetilde{c},\widetilde{d}\,)$. 
Indeed, we have $(e'\circ f)(\cdot,x_{2})\in \mathcal{L}^{1}([a,b])$ 
and $(e'\circ f)(x_{1},\cdot)\in\mathcal{C}^{1}((\widetilde{c},\widetilde{d}\,))$ as well as 
\begin{equation}\label{eq:fubini.1}
 (e'\circ F)(x_{2})=\int_{[a,b]}(e'\circ f)(x_{1},x_{2})\d x_{1}
\end{equation}
for every $x_{1}\in[a,b]$, $x_{2}\in (\widetilde{c},\widetilde{d}\,)$ and $e'\in E'$. 
Furthermore, for every $x_{2}\in(\widetilde{c},\widetilde{d}\,)$ there is $\varepsilon>0$ with 
$\overline{\mathbb{B}_{\varepsilon}(x_{2})}=[x_{2}-\varepsilon,x_{2}+\varepsilon]\subset(\widetilde{c},\widetilde{d}\,)$ and 
\[
C_{e'}:=\sup\bigl(|\partial_{x_{2}}(e'\circ f)(x_{1},\widetilde{x_{2}})|\;|\;
                  (x_{1},\widetilde{x_{2}})\in[a,b]\times\overline{\mathbb{B}_{\varepsilon}(x_{2})}\bigr)
<\infty,\quad e'\in E',
\]
because $e'\circ f\in\mathcal{C}^{1}(\Omega)$ for every $e'\in E'$. 
It follows from the scalar Leibniz rule for differentiation under the integral sign 
and the continuous dependency of a scalar integral on a parameter (see \cite[5.6, 5.7 Satz, p.\ 147-148]{elstrodt2005}) 
that $e'\circ F\in\mathcal{C}^{1}(\mathbb{B}_{\varepsilon}(x_{2}))$ for every $e'\in E'$. 
As $x_{2}\in(\widetilde{c},\widetilde{d}\,)$ is arbitrary, we 
get that $F$ is weakly $\mathcal{C}^{1}$ on $(\widetilde{c},\widetilde{d}\,)$. 
Due to \prettyref{lem:pettis.loc.complete} again, we deduce that 
$F$ is Pettis-integrable on $[c,d]$ and thus 
\begin{equation}\label{eq:fubini.2}
 \langle e', \int_{[c,d]} F(x_{2})\d x_{2}\rangle=\int_{[c,d]}\langle e',F(x_{2})\rangle\d x_{2}, \quad e'\in E'.
\end{equation}
Therefore we obtain for every $e'\in E'$ that
\begin{flalign*}
 &\quad\langle e', \int_{[a,b]\times[c,d]}f(x_{1},x_{2})\d (x_{1},x_{2})\rangle\\
 &= \int_{[a,b]\times[c,d]}\langle e',f(x_{1},x_{2})\rangle\d (x_{1},x_{2})
 = \int_{[c,d]}\int_{[a,b]}\langle e',f(x_{1},x_{2})\rangle \d x_{1}\d x_{2}\\
 &\;\;\mathclap{\underset{\eqref{eq:fubini.1}}{=}}\;\;\; \int_{[c,d]}\langle e',F(x_{2})\rangle \d x_{2}
 \underset{\eqref{eq:fubini.2}}{=} \langle e', \int_{[c,d]} F(x_{2})\d x_{2} \rangle 
 = \langle e', \int_{[c,d]}\int_{[a,b]}f(x_{1},x_{2})\d x_{1}\d x_{2}\rangle
\end{flalign*}
where we used the scalar version of Fubini's theorem in the second equation. The Hahn-Banach theorem yields the first equation 
from our claim and analogously we get the second equation
 \[
  \int_{[a,b]\times[c,d]}f(x_{1},x_{2})\d (x_{1},x_{2})
  =\int_{[a,b]}\int_{[c,d]}f(x_{1},x_{2})\d x_{2}\d x_{1}.
 \]
\end{proof}

Fubini's theorem for a continuous function $f\colon\Omega\subset\R^2\to E$ 
can also be found in \cite[Chap. 3, \S4.1, Remark, p.\ INT III.43]{bourbakiI} 
by Bourbaki under the restriction that $\oacx{(f([a,b]\times [c,d]))}$ is compact in $E$. 
From the condition that $f\colon\Omega\to E$ is weakly $\mathcal{C}^{1}$ follows that $f$ is continuous if $E$ is sequentially complete 
or more general if $E$ has metric ccp by \cite[6.4 Corollary, p.\ 19]{kruse2019_3}. 
Thus in this case one can also apply Bourbaki's version of Fubini's theorem.

\begin{rem}
Let $E$ be a locally complete lcHs over $\C$, $\Omega\subset\C^{d}$ open and $\gamma$ a $\mathcal{C}^{1}$-curve in $\Omega$. 
If $f\colon\Omega\to E$ is weakly $\gamma$-$\mathcal{C}^{1}$, then 
\[
\int_{\gamma}f(z)\d z=\int_{[a,b]}\cdots\int_{[a,b]}f(\gamma(t))\prod_{k=1}^{d}\gamma_{k}'(t_{k})\d t_{d}\cdots\d t_{1} 
\]
by Fubini's theorem.
\end{rem}

\begin{prop}[{chain rule}]\label{prop:chain.rule}
Let $E$ be an lcHs over $\C$, $\Omega\subset\C^{d}$ open, 
$\gamma\colon[a,b]^{d}\to\C^{d}$ a $\mathcal{C}^{1}$-curve in $\Omega$ and $F\colon\Omega\to E$
holomorphic. Then for every $1\leq j\leq d$
\begin{equation}\label{eq:chain.rule.1}
(\partial^{e_{j}}_{\R})^{\widehat{E}}\bigl((F\circ\phi^{-1})\circ (\phi\circ\gamma)\bigr)(t)
=(\partial^{e_{j}}_{\C})^{\widehat{E}}F(\gamma(t))\gamma_{j}'(t_{j}),\quad t\in [a,b]^{d}.
\end{equation}
\end{prop}
\begin{proof}
Due to \prettyref{rem:properties.of.holom.map} a)-c) and 
\prettyref{rem:compl-diff.is.real-diff} the map $F\circ\phi^{-1}\colon\phi(\Omega)\to E$ 
is continuous, the map $D_{\R}(F\circ\phi^{-1})(x)\colon\R^{2d}\to\widehat{E}$ 
is $\R$-linear and the map $D_{\R}(F\circ\phi^{-1})\colon
\phi(\Omega)\times\R^{2d}\to\widehat{E}$ is continuous. 
The set $\phi(\Omega)$ is open, thus for every $x\in\phi(\Omega)$ there is $R>0$ such that 
$\overline{\mathbb{B}_{R}(x)}\subset\phi(\Omega)$. Hence $D_{\R}(F\circ\phi^{-1})$ is uniformly 
continuous on the compact set $\overline{\mathbb{B}_{R}(x)}\times K$ for any compact set $K\subset\R^{2d}$. 
Let $(\widehat{E},(p_{\alpha})_{\alpha\in\widehat{\mathfrak{A}}})$ denote the completion of $E$. 
It follows that for every $\alpha\in\widehat{\mathfrak{A}}$ and $\varepsilon>0$ there is $\delta>0$ 
such that for all $y\in\overline{\mathbb{B}_{R}(x)}$ and $v\in K$ with $|y-x|=|(y,v)-(x,v)|<\delta$ 
we have
\[
\sup_{v\in K}p_{\alpha}(D_{\R}(F\circ\phi^{-1})(y)[v]-D_{\R}(F\circ\phi^{-1})(x)[v])<\varepsilon,
\]
implying that 
\[
D_{\R}(F\circ\phi^{-1})\colon\phi(\Omega)\to L_{c}(\R^{2d},\widehat{E})
\]
is continuous. 
Since $\gamma$ is a $\mathcal{C}^{1}$-curve in $\Omega$, there are an open set $X\subset\R^{d}$ with $[a,b]^{d}\subset X$
and a continuously partially differentiable extension $\widetilde{\gamma}$ of $\gamma$ on $X$ such that 
$\phi\circ\widetilde{\gamma}\colon X\to\R^{2d}$ is continuous, 
$D_{\R}(\phi\circ\widetilde{\gamma})\colon X\to L(\R^{d},\R^{2d})$ and 
by direct computation
\[
 D_{\R}(\phi\circ\widetilde{\gamma})(x)[v]=\sum_{k=1}^{d}\bigl((\operatorname{Re}\widetilde{\gamma}_{k})'(x_{k})e_{2k-1}
 +(\operatorname{Im}\widetilde{\gamma}_{k})'(x_{k})e_{2k}\bigr)v_{k},\quad x\in X,\, v\in \R^{d}.
\]
For $x,y\in X$ we set 
\[
 u_{k}:=(\operatorname{Re}\widetilde{\gamma}_{k})'(y_{k})-(\operatorname{Re}\widetilde{\gamma}_{k})'(x_{k})\quad\text{and}\quad
 w_{k}:=(\operatorname{Im}\widetilde{\gamma}_{k})'(y_{k})-(\operatorname{Im}\widetilde{\gamma}_{k})'(x_{k})
\]
and observe for $v\in\R^{d}$ that
\begin{flalign*} 
&\quad |D_{\R}(\phi\circ\widetilde{\gamma})(y)[v]-D_{\R}(\phi\circ\widetilde{\gamma})(x)[v]|\\
&=|\sum_{k=1}^{d}(u_{k}e_{2k-1}+w_{k}e_{2k})v_{k}|
=\bigl(\sum_{k=1}^{d}(u_{k}^{2}+w_{k}^{2})v_{k}^{2}\bigr)^{1/2}
\leq \sum_{k=1}^{d}(u_{k}^{2}+w_{k}^{2})^{1/2}|v_{k}|\\
&\leq |\sum_{k=1}^{d}(u_{k}^{2}+w_{k}^{2})^{1/2}e_{k}|\cdot|v|
\leq\sum_{k=1}^{d}|\widetilde{\gamma}_{k}'(y_{k})-\widetilde{\gamma}_{k}'(x_{k})|\cdot|v|
\end{flalign*}
where the second inequality follows from the Cauchy-Schwarz inequality. 
This implies that $D_{\R}(\phi\circ\widetilde{\gamma})\colon X\to L_{b}(\R^{d},\R^{2d})$ is continuous 
because $\gamma$ is continuously partially differentiable. 
Summarising, this means that $F\circ\phi^{-1}$ and $\phi\circ\widetilde{\gamma}$ are of class $\mathcal{C}^{1}_{k}$ in the 
notion of \cite[1.0.0 Definition, p.\ 59]{keller1974}.
From (the proof of) \cite[1.3.4 Corollary, p.\ 80]{keller1974} follows that 
\[
D_{\R}(F\circ\phi^{-1})(\phi\circ\widetilde{\gamma}(\cdot))[D_{\R}(\phi\circ\widetilde{\gamma})(\cdot)] 
\colon X\to L_{c}(\R^{d},\widehat{E})
\]
is continuous and thus 
\[
D_{\R}\bigl((F\circ\phi^{-1})\circ (\phi\circ\widetilde{\gamma})\bigr)(x)[v]=
D_{\R}(F\circ\phi^{-1})(\phi\circ\widetilde{\gamma}(x))[D_{\R}(\phi\circ\widetilde{\gamma})(x)[v]],
\;\;x\in X,\,v\in\R^{d},
\]
by the chain rule \cite[1.3.0 Theorem, p.\ 77]{keller1974}. 
In combination with \prettyref{rem:properties.of.holom.map} c) we obtain for every $x\in X$ and 
$v\in\R^{d}$ that
\begin{flalign*}
&\quad\;\sum_{k=1}^{d}(\partial^{e_{k}}_{\R})^{\widehat{E}}\bigl((F\circ\phi^{-1})\circ (\phi\circ\widetilde{\gamma})\bigr)(x)v_{k}\\
&=D_{\R}\bigl((F\circ\phi^{-1})\circ (\phi\circ\widetilde{\gamma})\bigr)(x)[v]
\;\;\mathclap{\underset{\eqref{eq:compl-diff.is.real-diff.1}}{=}}\;\;
 D_{\C}F\bigl((\phi^{-1}\circ\phi\circ\widetilde{\gamma})(x)\bigr)[\phi^{-1}(D_{\R}(\phi\circ\widetilde{\gamma})(x)[v])]\\
&= D_{\C}F(\widetilde{\gamma}(x))[\sum_{k=1}^{d}\widetilde{\gamma}_{k}'(x_{k})e_{k}v_{k}]
 =\sum_{k=1}^{d}(\partial^{e_{k}}_{\C})^{\widehat{E}}F(\widetilde{\gamma}(x))\widetilde{\gamma}_{k}'(x_{k})v_{k}
\end{flalign*} 
and thus with $v=e_{j}$, $1\leq j\leq d$,
\[
(\partial^{e_{j}}_{\R})^{\widehat{E}}\bigl((F\circ\phi^{-1})\circ (\phi\circ\widetilde{\gamma})\bigr)(x)
=(\partial^{e_{j}}_{\C})^{\widehat{E}}F(\widetilde{\gamma}(x))\widetilde{\gamma}_{j}'(x_{j}),
\]
connoting \eqref{eq:chain.rule.1} for $x\in[a,b]^{d}$.
\end{proof}

\begin{thm}[{fundamental theorem of calculus}]\label{thm:hauptsatz}
 Let $E$ be a locally complete lcHs over $\C$, $\Omega\subset\C$ open, 
 $\gamma\colon[a,b]\to \C$ a $\mathcal{C}^{1}$-curve in $\Omega$, $f\colon\Omega\to E$ weakly $\gamma$-$\mathcal{C}^{1}$ and let 
 there be a holomorphic function $F\colon\Omega\to E$ such that $F'=f$. Then 
 \begin{equation}\label{eq:hauptsatz.1}
  \int_{\gamma}f(z)\d z=F(\gamma(b))-F(\gamma(a)).
 \end{equation}
\end{thm}
\begin{proof}
The left-hand side of \eqref{eq:hauptsatz.1} is defined by \prettyref{prop:path.integral} a). 
Due to the chain rule \prettyref{prop:chain.rule} and \prettyref{rem:properties.of.holom.map} e) we have
\begin{align*}
 \int_{\gamma}f(z)\d z&=\int_{[a,b]}f(\gamma(t))\gamma'(t)\d t
 =\int_{[a,b]}\underbrace{f(\gamma(t))}_{=F'(\gamma(t))}\gamma'(t)\d t\\
 &\;\;\mathclap{\underset{\eqref{eq:chain.rule.1}}{=}}\;\;\;\int_{[a,b]}
\bigl((F\circ\phi^{-1})\circ (\phi\circ\gamma)\bigr)'(t)\d t.
\end{align*}
Looking at the last integral, we observe that
\begin{flalign*}
&\quad\langle e',\int_{[a,b]}\bigl((F\circ\phi^{-1})\circ (\phi\circ\gamma)\bigr)'(t)\d t\rangle\\
 &=\int_{a}^{b}\bigl(e'\circ(F\circ\phi^{-1})\circ (\phi\circ\gamma)\bigr)'(t)\d t\\
 &=\bigl(e'\circ(F\circ\phi^{-1})\circ (\phi\circ\gamma)\bigr)(b)-
 \bigl(e'\circ(F\circ\phi^{-1})\circ (\phi\circ\gamma)\bigr)(a)\\
 &=\langle e',(F\circ\gamma)(b)- (F\circ\gamma)(a)\rangle,\quad e'\in E',
\end{flalign*}
holds by the scalar fundamental theorem of calculus (applied to the real and the imaginary part of the integrand) 
where the second integral is a Riemann-integral. Finally, we deduce from the Hahn-Banach theorem that
\[
\int_{\gamma}f(z)\d z=F(\gamma(b))-F(\gamma(a)).
\]
\end{proof}

\begin{lem}[{Leibniz' rule}]\label{lem:compl.diff.under.int}
Let $E$ be a locally complete lcHs over $\C$, $V,U\subset\C^{d}$ open 
and $\gamma\colon[a,b]^{d}\to\C^{d}$ a $\mathcal{C}^{1}$-curve in $V$. 
\begin{enumerate}
\item [a)] Let $T$ be a set and $f,f_{n}\colon V\times T\to E$
 such that $f(\cdot,t),f_{n}(\cdot,t)\colon V\to E$ are weakly 
$\gamma$-$\mathcal{C}^{1}$ for every $t\in T$, $n\in\N$ and $f_{n}\to f$ uniformly on 
$\gamma([a,b]^{d})\times T$. 
Then 
\[
\lim_{n\to\infty}\int_{\gamma}f_{n}(z,t)\d z=\int_{\gamma}f(z,t)\d z
\]
holds uniformly on $T$.
\item [b)] Let $f\colon V\times U\to E$ be such that 
$f(\cdot,\lambda)\colon V\to E$ is weakly $\gamma$-$\mathcal{C}^{1}$ for every $\lambda\in U$, 
$f(z,\cdot)\colon U\to E$ is holomorphic for every $z\in V$ 
with $(\partial_{\lambda_{j}})^{E}f\colon V\times U\to E$ being continuous 
and $(\partial_{\lambda_{j}})^{E}f(\cdot,\lambda)\colon V\to E$ weakly $\gamma$-$\mathcal{C}^{1}$ 
for every $\lambda\in U$ and some $1\leq j\leq d$. Then 
\[
G\colon U\to E,\;G(\lambda):=\int_{\gamma}f(z,\lambda)\d z,
\]
is well-defined, complex differentiable with respect to $\lambda_{j}$ and 
\[
(\partial^{e_{j}}_{\C})^{E}G(\lambda)=\int_{\gamma}(\partial_{\lambda_{j}})^{E}f(z,\lambda)\d z\in E,
\quad \lambda\in U.
\] 
\end{enumerate}
\end{lem}
\begin{proof}
First, we remark that the integrals appearing in a) and b) are well-defined elements of $E$ 
by \prettyref{prop:path.integral} a) and the weakly $\gamma$-$\mathcal{C}^{1}$ condition.

a) Let $\alpha\in\mathfrak{A}$. Then we have 
\begin{flalign*}
&\quad \sup_{t\in T}p_{\alpha}\bigl(\int_{\gamma}f_{n}(z,t)\d z-\int_{\gamma}f(z,t)\d z\bigr)\\
&\leq l(\gamma)\sup_{(z,t)\in\gamma([a,b]^{d})\times T}p_{\alpha}(f_{n}(z,t)-f(z,t))\to 0,\quad n\to\infty,
\end{flalign*}
since $f_{n}\to f$ uniformly on $\gamma([a,b]^{d})\times T$.

b) Let $\lambda\in U$. Then there is $R>0$ such that $\overline{\mathbb{B}_{R}(\lambda)}\subset U$ as $U$ is open. 
 Let $(h_{n})$ be a null sequence in $\C\setminus\{0\}$ with $|h_{n}|<R/2$ for all $n\in\N$, which 
 implies that the line segment $\Gamma_{n}$ from $\lambda_{j}$ to $\lambda_{j}+h_{n}$ is a $\mathcal{C}^{1}$-curve in
 $\pi_{\lambda,j}^{-1}(\mathbb{B}_{R}(\lambda))$ that we parametrise by $[0,1]$. 
 Applying \prettyref{thm:hauptsatz} to the holomorphic function (in one variable)
 $f(z,\cdot)\circ\pi_{\lambda,j}\colon\pi_{\lambda,j}^{-1}(\mathbb{B}_{R}(\lambda))\to E$ for $z\in V$, we get 
 \begin{align*}
 f(z,\lambda+h_{n}e_{j})-f(z,\lambda)&=f(z,\pi_{\lambda,j}(\lambda_{j}+h_{n}))-f(z,\pi_{\lambda,j}(\lambda_{j}))\\
 &=\int_{\Gamma_{n}}(\partial_{\zeta_{j}})^{E}f(z,\pi_{\lambda,j}(\zeta_{j}))\d \zeta_{j}
 \end{align*}
 and therefore
 \begin{align*}
 |f_{n}(z,\lambda)|:=&\Bigl|\frac{f(z,\lambda+h_{n}e_{j})-f(z,\lambda)}{h_{n}}-(\partial_{\lambda_{j}})^{E}f(z,\lambda)\Bigr|\\
 =&\Bigl|\frac{1}{h_{n}}\int_{\Gamma_{n}}(\partial_{\zeta_{j}})^{E}f(z,\pi_{\lambda,j}(\zeta_{j}))
 -(\partial_{\lambda_{j}})^{E}f(z,\lambda) \d \zeta_{j}\Bigr|\\
 \leq & \frac{1}{|h_{n}|}l(\Gamma_{n})\sup_{\zeta_{j}\in\Gamma_{n}([0,1])}
 \bigl|(\partial_{\zeta_{j}})^{E}f(z,\pi_{\lambda,j}(\zeta_{j}))-(\partial_{\lambda_{j}})^{E}f(z,\lambda)\bigr|.
 \end{align*}
 Hence we obtain 
 \begin{flalign*}
 &\quad \sup_{z\in\gamma([a,b]^{d})}|f_{n}(z,\lambda)|\\
 &\leq \sup_{z\in\gamma([a,b]^{d})}\sup_{\zeta_{j}\in\Gamma_{n}([0,1])}
 \bigl|(\partial_{\zeta_{j}})^{E}f(z,\pi_{\lambda,j}(\zeta_{j}))-(\partial_{\lambda_{j}})^{E}f(z,\lambda)\bigr|
 \to 0,\quad n\to\infty,
 \end{flalign*}
 since $(\partial_{\lambda_{j}})^{E}f$ is uniformly continuous on the compact set 
 $\gamma([a,b]^{d})\times \overline{\mathbb{B}_{R}(\lambda)}$, 
 meaning $f_{n}\to 0$ uniformly on $\gamma([a,b]^{d})\times\{\lambda\}$.  
 From part a) we conclude $\int_{\gamma}f_{n}(z,\lambda)\d z\to 0$ and thus 
 \begin{align*}
 \int_{\gamma}(\partial_{\lambda_{j}})^{E}f(z,\lambda)\d z
 &=\lim_{n\to\infty}\int_{\gamma}\frac{f(z,\lambda+h_{n}e_{j})-f(z,\lambda)}{h_{n}}\d z\\
 &=\lim_{n\to\infty}\frac{G(z,\lambda+h_{n}e_{j})-G(z,\lambda)}{h_{n}}
 =(\partial^{e_{j}}_{\C})^{E}G(\lambda).
 \end{align*}
\end{proof}

\section{Holomorphic functions}

Now, we want to define complex partial derivatives of higher order for an $E$-valued function $f$. 
Let $E$ be an lcHs over $\C$ and $\Omega\subset\C^{d}$ open. 
A function $f\colon\Omega\to E$ is called complex partially differentiable on $\Omega$ 
and we write $f\in\mathcal{D}^{1}_{\C}(\Omega,E)$ 
if $\partial^{e_{j}}_{\C}f(z):=(\partial^{e_{j}}_{\C})^{E}f(z)\in E$ for every $z\in\Omega$ and $1\leq j\leq d$ 
(see Definition \ref{def:holom} c)).
For $k\in\N$, $k\geq 2$, a function $f$ is said to be $k$-times complex partially differentiable and
we write $f\in\mathcal{D}^{k}_{\C}(\Omega,E)$ if $f\in\mathcal{D}^{1}_{\C}(\Omega,E)$ and 
all its first complex partial derivatives are in $\mathcal{D}^{k-1}_{\C}(\Omega,E)$.
A function $f$ is called infinitely complex partially differentiable and we write $f\in\mathcal{D}^{\infty}_{\C}(\Omega,E)$
if $f\in\mathcal{D}^{k}_{\C}(\Omega,E)$ for every $k\in\N$.

Let $f\in\mathcal{D}^{k}_{\C}(\Omega,E)$. For $\beta=(\beta_{1},\ldots,\beta_{d})\in\N_{0}^{d}$ with 
$|\beta|:=\sum_{j=1}^{d}\beta_{j}\leq k$ we set 
$\partial^{\beta_{j}}_{\C}f:=(\partial^{\beta_{j}}_{\C})^{E}f:=f$, if $\beta_{j}=0$, and
\[
\partial^{\beta_{j}}_{\C}f:=(\partial^{\beta_{j}}_{\C})^{E}f
:=\underbrace{(\partial^{e_{j}}_{\C})\cdots(\partial^{e_{j}}_{\C})}_{\beta_{j}\text{-times}}f,
\]
if $\beta_{j}\neq 0$, as well as 
\[
\partial^{\beta}_{\C}f:=(\partial^{\beta}_{\C})^{E}f
:=(\partial^{\beta_{1}}_{\C})\cdots(\partial^{\beta_{d}}_{\C})f.
\]
A holomorphic function $f\colon\Omega\to E$ can be considered as a function from $\Omega$ to $\widehat{E}$,
which gives us $f\in\mathcal{D}^{1}_{\C}(\Omega,\widehat{E})$ (see \prettyref{rem:properties.of.holom.map} c)). 
Our goal is to show that we actually have $f\in\mathcal{D}^{\infty}_{\C}(\Omega,E)$ if $E$ is locally complete 
via proving Cauchy's integral formula for holomorphic functions.
For this purpose we recall the definition of a polydisc, its distinguished boundary and define integration along the
distinguished boundary.
For $w=(w_{1},\ldots,w_{d})\in\C^{d}$ and $R=(R_{1},\ldots,R_{d})\in(0,\infty]^{d}$ we define the \emph{polydisc} 
$\D_{R}(w):=\prod_{k=1}^{d}\mathbb{B}_{R_{k}}(w_{k})$ and its \emph{distinguished boundary} 
$\partial_{0}\D_{R}(w):=\prod_{k=1}^{d}\partial\mathbb{B}_{R_{k}}(w_{k})$. For $R,\rho\in(0,\infty]^{d}$ 
we write $\rho<R$ if $\rho_{k}<R_{k}$ for all $1\leq k\leq d$. 
For a function $f\colon\Omega\to E$ on a set $\Omega\subset\C^{d}$ with $\overline{\D_{\rho}(w)}\subset\Omega$ 
for some $w\in\C^{d}$ and $\rho\in(0,\infty)^{d}$ we set
\[
 \int_{\partial_{0}\D_{\rho}(w)}f(z)\d z:=\int_{\gamma}f(z)\d z
\]
if the integral on the right-hand side exists where $\gamma$ is the $\mathcal{C}^{1}$-curve in $\Omega$ given by 
the restriction $\gamma:=\widetilde{\gamma}_{\mid [0,2\pi]^{d}}$ of the map $\widetilde{\gamma}\colon\R^{d}\to\C^{d}$ 
defined by $\widetilde{\gamma}_{k}\colon\R\to\C$, $\widetilde{\gamma}_{k}(t):=w_{k}+\rho_{k}e^{it}$, for $1\leq k\leq d$.
Further, we need the usual notation
\[
\beta !:=\prod_{j=1}^{d}(\beta_{j}!)\quad\text{and}\quad (z-\zeta)^{\beta}:=\prod_{j=1}^{d}(z_{j}-\zeta_{j})^{\beta_{j}} 
\]
for $\beta\in\N_{0}^{d}$ and $z,\zeta\in\C^{d}$.

\begin{thm}[{Cauchy's integral formula}]\label{thm:cauchy.int.formula}
 Let $E$ be a locally complete lcHs over $\C$, $\Omega\subset\C^{d}$ open, $w\in\Omega$, $R\in(0,\infty]^{d}$
 with $\D_{R}(w)\subset\Omega$ and $f\colon\Omega\to E$ be holomorphic. 
 Then 
 \begin{equation}\label{eq:cauchy.int.formula.1}
  (\partial^{\beta}_{\C})f(\zeta)
 =\frac{\beta !}{(2\pi i)^{d}}\int_{\partial_{0}\D_{\rho}(w)}\frac{f(z)}{(z-\zeta)^{\beta+(1,\ldots,1)}}\d z\in E,
  \quad \zeta\in\D_{\rho}(w),\,\beta\in\N_{0}^{d},
 \end{equation}
 for all $\rho\in(0,\infty)^{d}$ with $\rho<R$.   
\end{thm}
\begin{proof}
 Let $\widetilde{\gamma}$ and $\gamma$ be defined as above for $\partial_{0}\D_{\rho}(w)$. First, we consider the case $\beta=0$. 
 We set
 \[
 g_{\zeta}\colon\Omega\setminus\{\zeta\}\to E,\;
 g_{\zeta}(z):=\frac{f(z)}{(z-\zeta)^{(1,\ldots,1)}},
 \]
 for $\zeta\in\D_{\rho}(w)$ and observe that 
 $g_{\zeta}\circ\widetilde{\gamma}$ is weakly $\mathcal{C}^{1}$ on $\R^{d}$ 
 since $f$ is holomorphic on $\Omega$ and $\widetilde{\gamma}\in\mathcal{C}^{1}(\R^{d},\C^{d})$. 
 Thus $g_{\zeta}$ is weakly $\gamma$-$\mathcal{C}^{1}$ and integrable along $\gamma$ by \prettyref{prop:path.integral} a).
 Since $f$ is weakly holomorphic and $g_{\zeta}$ integrable along $\gamma$, we get by the scalar version of 
 Cauchy's integral formula that 
 \begin{align*}
 (e'\circ f)(\zeta)&=\frac{1}{(2\pi i)^{d}}\int_{\partial_{0}\D_{\rho}(w)}\frac{(e'\circ f)(z)}{(z-\zeta)^{(1,\ldots,1)}}\d z\\
 &=\langle e', \frac{1}{(2\pi i)^{d}}\int_{\partial_{0}\D_{\rho}(w)}\frac{f(z)}{(z-\zeta)^{(1,\ldots,1)}}\d z\rangle,\quad e'\in E',
 \end{align*}
 implying
 \[
 f(\zeta)=\frac{1}{(2\pi i)^{d}}\int_{\partial_{0}\D_{\rho}(w)}\frac{f(z)}{(z-\zeta)^{(1,\ldots,1)}}\d z
 \]
 by the Hahn-Banach theorem, which proves \eqref{eq:cauchy.int.formula.1} for $\beta=0$.
 
 Let $n\in\N_{0}$ and \eqref{eq:cauchy.int.formula.1} be fulfilled for every $\beta\in\N_{0}^{d}$ with $|\beta|=n$. 
 Let $\beta\in\N_{0}^{d}$ with $|\beta|=n+1$. Then there is $j\in\N$, $1\leq j\leq d$, and $\widetilde{\beta}\in\N_{0}^{d}$ with 
 $|\widetilde{\beta}|=n$ such that $\beta=\widetilde{\beta}+e_{j}$. 
 Let $\zeta\in\D_{\rho}(w)$. Then there is 
 $0<r<\rho$ such that $\zeta\in\D_{r}(w)$. 
 We define the open set $V:=\D_{R}(w)\setminus\overline{\D_{r}(w)}$ and the function  
 \[
 F_{\widetilde{\beta}}\colon V\times \D_{r}(w) \to E,\;
 F_{\widetilde{\beta}}(z,\lambda):=\frac{f(z)}{(z-\lambda)^{\widetilde{\beta}+(1,\ldots,1)}}.
 \]
 Further, we compute for $\lambda\in\D_{r}(w)$
 \[
 \partial_{\lambda_{j}}F_{\widetilde{\beta}}(z,\lambda)
 =\frac{(\widetilde{\beta}_{j}+1)f(z)}{(z-\lambda)^{\widetilde{\beta}+e_{j}+(1,\ldots,1)}}
 =\frac{\beta_{j}f(z)}{(z-\lambda)^{\beta+(1,\ldots,1)}}\in E, \quad z\in V.
 \]
 We see that $\gamma$ is a $\mathcal{C}^{1}$-curve in $V$ and 
 $F_{\widetilde{\beta}}(\cdot,\lambda)\circ\widetilde{\gamma}$ and 
 $\partial_{\lambda_{j}}F_{\widetilde{\beta}}(\cdot,\lambda)\circ\widetilde{\gamma}$ are weakly 
 $\mathcal{C}^{1}$ on $\R^{d}$ for every $\lambda\in\D_{r}(w)$ since $f$ is holomorphic 
 on $\Omega$. Hence $F_{\widetilde{\beta}}(\cdot,\lambda)$ and 
 $\partial_{\lambda_{j}}F_{\widetilde{\beta}}(\cdot,\lambda)$ 
 are weakly $\gamma$-$\mathcal{C}^{1}$ for every $\lambda\in\D_{r}(w)$. 
 In addition, $\partial_{\lambda_{j}}F_{\widetilde{\beta}}$ is continuous 
 on $V\times\D_{r}(w)$ by \prettyref{rem:properties.of.holom.map} b),
 $F_{\widetilde{\beta}}(z,\cdot)$ is holomorphic on $\D_{r}(w)$ for every $z\in V$ 
 and thus we can apply Leibniz' rule \prettyref{lem:compl.diff.under.int} b), yielding
 \[
  \partial^{e_{j}}_{\C}(\partial^{\widetilde{\beta}}_{\C}f)(\lambda)
  =\frac{\widetilde{\beta}!}{(2\pi i)^{d}}\int_{\gamma}\partial_{\lambda_{j}}F_{\widetilde{\beta}}(z,\lambda)\d z
  =\frac{\beta!}{(2\pi i)^{d}}\int_{\gamma}\frac{f(z)}{(z-\lambda)^{\beta+(1,\ldots,1)}}\d z\in E
 \]
 for every $\lambda\in\D_{r}(w)$, in particular for $\lambda=\zeta$, 
 where we used the induction hypothesis in the first equation. It remains to be shown that  
 $\partial^{e_{j}}_{\C}(\partial^{\widetilde{\beta}}_{\C}f)(\lambda)=\partial^{\beta}_{\C}f(\lambda)$ for every $\lambda\in\D_{r}(w)$, 
 i.e.\ that the order of the partial derivatives does not matter. 
 For $\widetilde{\beta}=0$ this is clear. If $|\widetilde{\beta}|=1$, then our 
 preceding considerations imply that 
 \[
\partial^{e_{j}}_{\C}\partial^{e_{k}}_{\C}f(\lambda)=\frac{1}{(2\pi i)^{d}}
\int_{\partial_{0}\D_{\rho}(w)}\frac{f(z)}{(z-\lambda)^{e_{j}+e_{k}+(1,\ldots,1)}}\d z
=\partial^{e_{k}}_{\C}\partial^{e_{j}}_{\C}f(\lambda)
\]
for all $1\leq j,k\leq d$. 
This yields that $\partial^{e_{j}}_{\C}(\partial^{\widetilde{\beta}}_{\C}f)(\lambda)=\partial^{\beta}_{\C}f(\lambda)$ 
for every $\lambda\in\D_{r}(w)$
\end{proof}

Cauchy's integral formula for derivatives is usually derived by using the Riemann-integral instead of the Pettis-integral 
and can be found for holomorphic functions in one variable in \cite[Th\'eor\`eme 1, p.\ 37-38]{Grothendieck1953}, 
in several variables in \cite[Corollary 3.7, p.\ 85]{bochnak1971} and infinitely many variables in 
\cite[Proposition 2.4, p.\ 55]{dineen1981} as well. 
The Riemann-integrals are elements of $E$ under the 
condition that $E$ has ccp in \cite{Grothendieck1953} 
or more general if $E$ is sequentially complete in \cite{bochnak1971} and \cite{dineen1981} by 
\cite[Lemma 1.1, p.\ 79]{bochnak1971}. In general, they are only elements of the completion $\widehat{E}$. 
From our approach using Pettis-integrals we guarantee that they belong to $E$ even if $E$ is only locally complete. 

\begin{cor}\label{cor:compl.deriv.commute}
If $E$ is a locally complete lcHs over $\C$, $\Omega\subset\C^{d}$ open 
and $f\colon\Omega\to E$ holomorphic, then $(\partial^{\beta}_{\C})f$ 
does not depend on the order of the partial derivatives involved and $f\in\mathcal{D}^{\infty}_{\C}(\Omega,E)$.
\end{cor}
\begin{proof}
The independence of the order follows from the proof of Cauchy's integral formula \prettyref{thm:cauchy.int.formula}. 
Combining this formula with the commutativity of the complex 
partial derivatives, we conclude $f\in\mathcal{D}^{\infty}_{\C}(\Omega,E)$.
\end{proof}

For an lcHs $E$ over $\C$, an open set $\Omega\subset\C^{d}$ and a 
function $f\colon\Omega\to E$, we write $f\in\mathcal{C}^{k}_{\R}(\Omega,E)$ if 
$f\circ\phi^{-1}\in\mathcal{C}^{k}(\phi(\Omega),E)$ for $k\in\N_{0,\infty}$. 
We define the space 
\[
\mathcal{O}(\Omega,E):=\bigl\{f\in\mathcal{C}^{1}_{\R}(\Omega,E)\;|\;\forall\;\beta\in\N_{0}^{d},\,z\in\Omega:\;
(\partial^{\beta}_{\C})^{E}f(z)\in E\bigr\},
\]
which we equip with the system of seminorms given by 
\[
 |f|_{K,\alpha}:=\sup_{z\in K}p_{\alpha}(f(z)),\quad f\in \mathcal{O}(\Omega,E),
\]
for $K\subset\Omega$ compact and $\alpha\in\mathfrak{A}$. If $E=\C$, we just write $\mathcal{O}(\Omega):=\mathcal{O}(\Omega,\C)$.

Due to Cauchy's integral formula and \prettyref{rem:compl-diff.is.real-diff} in combination with 
\prettyref{rem:properties.of.holom.map} c)+e), we already know 
that every holomorphic function $f\colon\Omega\to E$ is an element of $\mathcal{O}(\Omega,E)$ and 
we prove in the following that every element of $\mathcal{O}(\Omega,E)$ is a holomorphic function on $\Omega$ 
as well if $E$ is locally complete. 
The space $\mathcal{O}(\Omega)$ coincides with the space of all $\C$-valued holomorphic functions on
$\Omega$ in the sense of \cite[Definition 1.7.1, p.\ 47]{pflug2008} by \cite[Theorem 1.7.6, p.\ 48-49]{pflug2008} 
and is a Fr\'{e}chet space by \cite[Example 1.10.7 (a), p.\ 66]{pflug2008}. 
As a start in proving that $\mathcal{O}(\Omega,E)$ is the space of all holomorphic functions from $\Omega$ 
to a locally complete space $E$, we show that the elements of $\mathcal{O}(\Omega,E)$ fulfil the Cauchy inequality.

\begin{cor}[{Cauchy inequality}]\label{cor:cauchy.ineq}
Let $E$ be a locally complete lcHs over $\C$ and $\Omega\subset\C^{d}$ open.
\begin{enumerate}
\item [a)] If $w\in\Omega$, $R\in(0,\infty]^{d}$ with $\D_{R}(w)\subset\Omega$ 
and $f\in\mathcal{O}(\Omega,E)$, then 
 \begin{equation}\label{eq:cauchy.ineq.1}
  p_{\alpha}(\partial^{\beta}_{\C}f(\zeta))
  \leq \frac{\beta !}{\rho^{\beta}}\max_{z\in\partial_{0}\D_{\rho}(w)}p_{\alpha}(f(z)),
  \quad \zeta\in\D_{\rho}(w),\, \beta\in\N_{0}^{d},
 \end{equation}
for every $\rho\in(0,\infty)^{d}$ with $\rho<R$ and $\alpha\in\mathfrak{A}$.
\item [b)] For every compact set $K\subset\Omega$ there is a compact set $K'\subset \Omega$ such that 
for every $\beta\in\N_{0}^{d}$ there is $C_{K,\beta}>0$ such that for every $\alpha\in\mathfrak{A}$ 
and every $f\in\mathcal{O}(\Omega,E)$ it holds that
 \begin{equation}\label{eq:cauchy.ineq.2}
 \sup_{z\in K}p_{\alpha}(\partial^{\beta}_{\C}f(z))\leq C_{K,\beta}\max_{z\in K'}p_{\alpha}(f(z)).
 \end{equation}
\end{enumerate}
\end{cor}
\begin{proof}
a) For $\alpha\in\mathfrak{A}$ we set $B_{\alpha}:=\{x\in E\;|\;p_{\alpha}(x)<1\}$, its polar 
$B_{\alpha}^{\circ}:=\{e'\in E'\;|\;\forall x\in B_{\alpha}:\;|e'(x)|\leq 1\}$ and
denote by $\gamma$ the $\mathcal{C}^{1}$-curve on $[0,2\pi]^{d}$ corresponding to $\partial_{0}\D_{\rho}(w)$. 
It follows from the scalar version of Cauchy's integral formula (see \cite[Theorem 1.7.6, p.\ 48-49]{pflug2008}) 
that for all $\zeta\in\D_{\rho}(w)$ and $\beta\in\N_{0}^{d}$ we have
\begin{flalign*}
 &\;\quad p_{\alpha}\bigl((\partial^{\beta}_{\C})^{E}f(\zeta)\bigr)\\
 &=\sup_{e'\in B_{\alpha}^{\circ}}\bigl|(\partial^{\beta}_{\C})^{\C}(e'\circ f)(\zeta)\bigr|
 =\frac{\beta !}{(2\pi)^{d}}\sup_{e'\in B_{\alpha}^{\circ}}\bigl|
 \int_{\partial_{0}\D_{\rho}(w)}\frac{e'(f(z))}{(z-\zeta)^{\beta+(1,\ldots,1)}}\d z\bigr|\\
 &\leq \frac{\beta !}{(2\pi)^{d}} \underbrace{l(\gamma)}_{=(2\pi)^{d}\rho^{(1,\ldots,1)}}\sup_{e'\in B_{\alpha}^{\circ}}
 \sup_{z\in\partial_{0}\D_{\rho}(w)}\frac{|e'(f(z))|}{\rho^{\beta+(1,\ldots,1)}}
 =\frac{\beta !}{\rho^{\beta}}\sup_{z\in\partial_{0}\D_{\rho}(w)}p_{\alpha}(f(z))
\end{flalign*}
where we used \cite[Proposition 22.14, p.\ 256]{meisevogt1997} in the first and last equation 
to get from $p_{\alpha}$ to $\sup_{e'\in B_{\alpha}^{\circ}}$ and back. Our statement 
follows from the continuity of $f$ on $\Omega$ by \prettyref{rem:properties.of.holom.map} b) and the compactness
of the distinguished boundary.

b) Is a direct consequence of a) since every compact set $K\subset\Omega$ can be covered 
by a finite number $n$ of open, bounded polydiscs $\D_{\rho_{j}}(w_{j})$ with 
$\overline{\D_{\rho_{j}}(w_{j})}\subset \Omega$ for $1\leq j\leq n$.
\end{proof}

For sequentially complete $E$ Cauchy's inequality can also be found in \cite[Proposition 2.5, p.\ 57]{dineen1981} 
and as a direct consequence we obtain:

\begin{rem}[{Weierstrass}]\label{rem:equiv.top}
Let $E$ be a locally complete lcHs over $\C$ and $\Omega\subset\C^{d}$ open. 
Then the system of seminorms generated by 
\[
 |f|_{K,m,\alpha}:=\sup_{\substack{z\in K\\\beta\in\N_{0}^{d},|\beta|\leq m}}
 p_{\alpha}\bigl((\partial^{\beta}_{\C})^{E}f(z)\bigr),\quad f\in \mathcal{O}(\Omega,E),
\]
for $K\subset\Omega$ compact, $m\in\N_{0}$ and $\alpha\in\mathfrak{A}$ 
induces the same topology on $\mathcal{O}(\Omega,E)$ as the system $(|f|_{K,\alpha})$ by 
\eqref{eq:cauchy.ineq.2}.
\end{rem}

This remark implies \cite[Proposition 3.1, p.\ 85]{bochnak1971} if $E$ is sequentially complete. 
We observe the following useful relation between real and complex first partial derivatives.

\begin{prop}\label{prop:first.real.compl.part.deriv}
If $E$ is an lcHs over $\C$, $\Omega\subset\C^{d}$ open 
and $f\in\mathcal{O}(\Omega,E)$, then for every $1\leq j\leq d$ and $x\in\phi(\Omega)$
\[
\partial^{e_{2j}}_{\R}(f\circ \phi^{-1})(x)=i
\partial^{e_{j}}_{\C}f(\phi^{-1}(x))
\quad\text{and}\quad 
\partial^{e_{2j-1}}_{\R}(f\circ \phi^{-1})(x)=
\partial^{e_{j}}_{\C}f(\phi^{-1}(x)).
\]
\end{prop}
\begin{proof}
$f\in\mathcal{C}^{1}_{\R}(\Omega,E)$ and for $x=(x_{1},\ldots,x_{2d})\in\phi(\Omega)$ we get
\begin{align*}
  \partial^{e_{2j}}_{\R}(f\circ\phi^{-1})(x)
&=\lim_{\substack{h\to 0\\h\in\R,h\neq 0}}\frac{f(\ldots,x_{j-1}+ix_{j}+ih,\ldots)-f(\ldots,x_{j-1}+ix_{j},\ldots)}{h}\\
&=i\lim_{\substack{h\to 0\\h\in\R,h\neq 0}}\frac{f(\ldots,x_{j-1}+ix_{j}+ih,\ldots)-f(\ldots,x_{j-1}+ix_{j},\ldots)}{ih}\\
&=i\partial^{e_{j}}_{\C}f(\phi^{-1}(x))
\end{align*}
as well as
\begin{align*}
  \partial^{e_{2j-1}}_{\R}(f\circ\phi^{-1})(x)
&=\lim_{\substack{h\to 0\\h\in\R,h\neq 0}}\frac{f(\ldots,x_{j}+ix_{j+1}+h,\ldots)-f(\ldots,x_{j}+ix_{j+1},\ldots)}{h}\\
&=\partial^{e_{j}}_{\C}f(\phi^{-1}(x)).
\end{align*}
\end{proof}

\begin{prop}\label{prop:eps-prod-holom}
Let $E$ be a locally complete lcHs over $\C$ and $\Omega\subset\C^{d}$ open. 
Then the map 
\[
S\colon\mathcal{O}(\Omega)\varepsilon E\to\mathcal{O}(\Omega,E),\;u\longmapsto [z\mapsto u(\delta_{z})],
\]
is a (topological) isomorphism where $\delta_{z}$ is the point evaluation functional at $z$.
\end{prop}
\begin{proof}
Let $u\in\mathcal{O}(\Omega)\varepsilon E$. 
Due to \cite[Proposition 10, p.\ 1520]{kruse2020} and the barrelledness of the Fr\'echet space $\mathcal{O}(\Omega)$ we have 
$(\partial^{\beta}_{\C})^{E}S(u)(z)=u(\delta_{z}\circ\partial^{\beta}_{\C})\in E$ 
for all $\beta\in\N_{0}^{d}$ and $z\in\Omega$ where one has to replace $\partial^{\beta}_{\R}$ by $\partial^{\beta}_{\C}$
and the space $\mathcal{CW}^{k}(\Omega)$ by $\mathcal{O}(\Omega)$ in the proof of \cite[Proposition 10, p.\ 1520]{kruse2020}.
Furthermore, $S(u)\in \mathcal{C}^{1}_{\R}(\Omega,E)$ by \cite[Proposition 10, p.\ 1520]{kruse2020} with $k=1$, 
\prettyref{rem:equiv.top} and \prettyref{prop:first.real.compl.part.deriv}, which implies that $S(u)\in\mathcal{O}(\Omega,E)$. 

Let $f\in\mathcal{O}(\Omega,E)$ and $K\subset\Omega$ be compact. 
It is easily checked that $e'\circ f\in\mathcal{O}(\Omega)$ for every $e'\in E'$.
Since $f\circ\phi^{-1}$ is weakly $\mathcal{C}^{1}$ on the open set $\phi(\Omega)\subset\R^{2d}$, 
it follows from \cite[Proposition 2, p.\ 354]{Bonet2002} that 
$K_{1}:=\overline{\operatorname{acx}}(f(K))$ is a absolutely convex and compact.
The inclusion $N_{K}(f):=f(K)\subset K_{1}$ implies that $S$ is a (topological) isomorphism by \cite[Theorem 14 (iii), p.\ 1524]{kruse2020}. 
\end{proof}

Once we have the equivalent conditions for holomorphy from our main \prettyref{thm:holom.equiv}, 
namely the equivalence `$a)\Leftrightarrow d)$', the preceding 
proposition is just a consequence of \cite[Theorem 9, p.\ 232]{B/F/J}.

\begin{thm}\label{thm:powerseries}
  Let $E$ be a locally complete lcHs over $\C$, $z\in\C^{d}$ and $R\in(0,\infty]^{d}$. 
  Then the tensor product $\mathcal{O}(\D_{R}(z))\otimes E$ is sequentially dense in $\mathcal{O}(\D_{R}(z),E)$ and
   \[
    f=\sum_{\beta\in\N_{0}^{d}}\frac{(\partial^{\beta}_{\C})^{E}f(z)}{\beta!}(\cdot - z)^{\beta}
   \]
   for all $f\in\mathcal{O}(\D_{R}(z),E)$ where the series converges in $\mathcal{O}(\D_{R}(z),E)$.
\end{thm}
\begin{proof}
The monomials $(\cdot - z)^{\beta}$, $\beta\in\N_{0}^{d}$, form an equicontinuous Schauder basis with 
associated coefficient functionals $\tfrac{1}{\beta!}(\delta_{z}\circ\partial^{\beta}_{\C})$ 
of the barrelled space $\mathcal{O}(\D_{R}(z))$ by \cite[Theorem 1.7.6, p.\ 48-49]{pflug2008}. 
Thus our statement follows from \prettyref{prop:eps-prod-holom} and \cite[3.6 Corollary, p.\ 7]{kruse2018_1}.  
\end{proof}

In the one variable case the theorem above is given in \cite[3.6 Corollary c), p.\ 7]{kruse2018_1} combined with
\cite[4.5 Theorem, p.\ 13]{kruse2018_1} as well. 

\begin{cor}\label{cor:hol.in.O}
Let $E$ be a locally complete lcHs over $\C$ and $\Omega\subset\C^{d}$ open. 
Then the following statements are equivalent for a function $f\colon\Omega\to E$.
\begin{enumerate}
 \item [a)] $f$ is holomorphic on $\Omega$.
 \item [b)] $f\in\mathcal{O}(\Omega,E)$.
\end{enumerate}
\end{cor}
\begin{proof}
We only need to prove the implication `$b)\Rightarrow a)$'. We claim that for every $z\in\Omega$ it holds that
\begin{equation}\label{eq:hol.in.O.1}
df(z)[v]=\sum_{j=1}^{d}\partial^{e_{j}}_{\C}f(z)v_{j},\quad v\in\C^{d}.
\end{equation}
Observe that the right-hand side is already linear in $v$.
Let $\alpha\in\mathfrak{A}$ and $z\in\Omega$. Then there is 
$R\in (0,\infty]^{d}$ such that $\D_{R}(z)\subset \Omega$. We fix 
$\rho\in(0,\infty)^{d}$ with $\rho < R$ and derive from \prettyref{thm:powerseries} a) 
that 
\begin{align*}
g(w,z):=& f(w)-f(z)-\sum_{j=1}^{d}\partial^{e_{j}}_{\C}f(z)(w_{j}-z_{j})\\
=&\sum_{|\beta|>1}\frac{\partial^{\beta}_{\C}f(z)}{\beta!}(w - z)^{\beta}
=\sum_{j=1}^{d}(w_{j} - z_{j})\sum_{|\beta|\geq 1}
\frac{\partial^{\beta+e_{j}}_{\C}f(z)}{(\beta+e_{j})!}(w - z)^{\beta}
\end{align*}
for every $w\in\overline{\D_{\rho}(z)}$.

Let $0<\varepsilon\leq 1$ and set $r:=\min_{1\leq j\leq d} \rho_{j}$. We observe that 
$\mathbb{B}_{\varepsilon\cdot r/2}(z)$ is a subset of $\D_{\rho}(z)$. 
Applying Cauchy's inequality \eqref{eq:cauchy.ineq.1}, we obtain for $w\in\mathbb{B}_{\varepsilon\cdot r/2}(z)$
\begin{flalign*}
&\quad p_{\alpha}\Bigl(\frac{\partial^{\beta+e_{j}}_{\C}f(z)}{(\beta+e_{j})!}(w - z)^{\beta}\Bigr)\\
&\leq \frac{\prod_{k=1}^{d}|w_{k}-z_{k}|^{\beta_{k}}}{\rho^{\beta+e_{j}}} \max_{\zeta\in\partial_{0}\D_{\rho}(z)} p_{\alpha}(f(\zeta))
\leq \frac{1}{r}\max_{\zeta\in\partial_{0}\D_{\rho}(z)}p_{\alpha}(f(\zeta))
\prod_{k=1}^{d}\Bigl(\frac{|w-z|}{r}\Bigr)^{\beta_{k}}\\
&\leq \frac{1}{r}\max_{\zeta\in\partial_{0}\D_{\rho}(z)}p_{\alpha}(f(\zeta))
\prod_{k=1}^{d}\varepsilon^{\beta_{k}}\frac{1}{2^{\beta_{k}}}
\leq \frac{\varepsilon^{d}}{2^{|\beta|}r}\max_{\zeta\in\partial_{0}\D_{\rho}(z)}p_{\alpha}(f(\zeta)).
\end{flalign*}
Hence we conclude for every $w\in\mathbb{B}_{\varepsilon\cdot r/2}(z)$
\[
    p_{\alpha}\Bigl(\frac{g(w,z)}{|w-z|}\Bigr)
\leq\frac{d\cdot\varepsilon^{d}}{r}\max_{\zeta\in\partial_{0}\D_{\rho}(z)}p_{\alpha}(f(\zeta))
    \sum_{|\beta|\geq 1}\frac{1}{2^{|\beta|}}
\leq\frac{2^{d}\cdot d\cdot\varepsilon^{d}}{r}\max_{\zeta\in\partial_{0}\D_{\rho}(z)}p_{\alpha}(f(\zeta))
\]
where the last estimate follows from \cite[Corollary 1.2.14 (a), p.\ 12-13]{pflug2008}. 
Letting $\varepsilon\to 0$, proves \eqref{eq:hol.in.O.1}.

Fix $v\in\C^{d}$ and let $z,w\in\Omega$. The estimate 
\[
p_{\alpha}(df(w)[v]-df(z)[v])
\leq \sum_{j=1}^{d}p_{\alpha}\bigl(\partial^{e_{j}}_{\C}f(w)-\partial^{e_{j}}_{\C}f(z)\bigr)|v_{j}|
\]
implies that $df(\cdot)[v]$ is continuous on $\Omega$ since $f\in\mathcal{C}^{1}_{\R}(\Omega)$ 
and by \prettyref{prop:first.real.compl.part.deriv}. Therefore $f$ is holomorphic on $\Omega$. 
\end{proof}

\section{The main theorem}

We briefly recall the following definitions which enable us to phrase our main theorem 
concerning holomorphic functions in several variables. 
Let $E$ be an lcHs over $\C$. For an open set $\Omega\subset\R^{2d}$ and $1\leq j\leq d$ we define 
the \emph{Cauchy-Riemann operator} by
\[
\overline{\partial}_{j}f(x):=(\overline{\partial}_{j})^{E}f(x)
:=\frac{1}{2}(\partial^{e_{2j-1}}_{\R}+i\partial^{e_{2j}}_{\R})f(x),
\quad f\in\mathcal{C}^{1}(\Omega,E),\,x\in\Omega.
\]
A function $f\colon\Omega\to E$ from a topological space $\Omega$ to $E$ is called \emph{locally bounded} 
on a subset $\Lambda\subset\Omega$ if for every $z\in\Lambda$ there is a neighbourhood $U\subset\Omega$ of $z$ 
such that $f$ is bounded on $U$.
A subspace $G\subset E'$ is said to be \emph{separating} if for every $x,y\in E$ there is  
$e'\in G$ such that $e'(x)\neq e'(y)$. 
A subspace $G\subset E'$ is said to \emph{determine boundedness} if every $\sigma(E,G)$-bounded subset of $E$ is already bounded 
where $\sigma(E,G)$ denotes the weak topology w.r.t.\ the dual pair $\langle E,G\rangle$.
If $G$ determines boundedness, then $G$ is separating. For instance, $G:=E'$ determines boundedness by Mackey's theorem 
and more examples are given in \cite[Remark 1.4, p.\ 781]{Arendt2000} and \cite[Remark 11, p.\ 233]{B/F/J}.

\begin{thm}\label{thm:holom.equiv}
 Let $E$ be a locally complete lcHs over $\C$ and $\Omega\subset\C^{d}$ be open. 
 Then the following statements are equivalent for a function $f\colon\Omega\to E$.
 \begin{enumerate}
 \item [a)] $f$ is holomorphic (G\^{a}teaux-, separately holomorphic) on $\Omega$.
 \item [b)] $f\in\mathcal{C}(\Omega,E)$ and $\partial^{e_{j}}_{\C}f(z)$ exists in $\widehat{E}$ 
 for every $z\in\Omega$ and $1\leq j\leq d$.
 \item [c)] $\partial^{e_{j}}_{\C}f(z)$ exists in $\widehat{E}$ for every $z\in\Omega$ and $1\leq j\leq d$.
 \item [d)] $f\in\mathcal{C}^{\infty}_{\R}(\Omega,E)$ and $\overline{\partial}_{j}(f\circ\phi^{-1})=0$ 
 for all $1\leq j\leq d$.
 \item [e)] There is a subspace $G\subset E'$ which determines boundedness such that $e'\circ f$ is holomorphic 
 (G\^{a}teaux-, separately holomorphic) on $\Omega$ for every $e'\in G$. 
 \item [f)] $f$ is locally bounded outside some compact set $K\subset\Omega$ 
 and there is a separating subspace $G\subset E'$ such that 
 $e'\circ f$ is holomorphic (G\^{a}teaux-, separately holomorphic) on $\Omega$ for every $e'\in G$.
 \item [g)] For every $z\in\Omega$ there are $R\in(0,\infty]^{d}$ and 
 $(a_{\beta})_{\beta\in\N_{0}^{d}}\subset E$ such that 
 \[
  f=\sum_{\beta\in\N_{0}^{d}}a_{\beta}(\cdot - z)^{\beta} \quad\text{on}\;\D_{R}(z).
 \]
 \end{enumerate}
 If one of the equivalent conditions above is fulfilled, then 
 \begin{equation}\label{eq:holom.equiv.1}
  df(z)[v]=Df(z)[v]=\sum_{j=1}^{d}\partial^{e_{j}}_{\C}f(z)v_{j}\in E,\quad z\in\Omega,\,v\in\C^{d}.
 \end{equation}
\end{thm}
\begin{proof}
We write $x)(i)$ if we consider $x)$ for holomorphic functions, 
$x)(ii)$ if we consider $x)$ for G\^{a}teaux-holomorphic functions and 
$x)(iii)$ if we consider $x)$ for separately holomorphic functions in 
the cases $x\in\{a,e,f\}$. First, we remark that the endorsement \eqref{eq:holom.equiv.1} 
follows from Cauchy's integral formula and \prettyref{rem:properties.of.holom.map} c)+e).\\
`$a)(i)\Leftrightarrow e)(i)$': The implication `$\Rightarrow$' is clear with $G:=E'$. 
Let us turn to `$\Leftarrow$'. We claim that $\mathcal{O}(\Omega,E)$ coincides with the space of 
functions $f\colon\Omega\to E$ such that $e'\circ f\in\mathcal{O}(\Omega)$ for each $e'\in G$, 
which then yields the desired equivalence by \prettyref{cor:hol.in.O}. 
$\mathcal{O}(\Omega)$ is a closed subspace of $\mathcal{C}^{\infty}(\phi(\Omega))$ via the map $f\mapsto f\circ \phi^{-1}$ 
(see e.g.\ \cite[p.\ 691]{F/J/W}). Thus we can apply the weak-strong principle \cite[Corollary 10 (a), p.\ 233]{B/F/J} in 
combination with \cite[Definition 3, p.\ 229-230]{B/F/J} and \prettyref{prop:eps-prod-holom}, proving our claim. \\
`$a)(i)\Rightarrow g)$': Follows from \prettyref{cor:hol.in.O} and \prettyref{thm:powerseries} since
for every $f\in\mathcal{O}(\Omega,E)$ there is $z\in\Omega$ and $R\in(0,\infty]^{d}$ such 
that $f_{\mid \D_{R}(z)}\in\mathcal{O}(\D_{R}(z),E)$.\\
`$g)\Rightarrow e)(i)$': Let $z\in\Omega$, $R\in(0,\infty]^{d}$ and 
$(a_{\beta})_{\beta\in\N_{0}^{d}}\subset E$ be such that 
\[
 f(w)=\sum_{\beta\in\N_{0}^{d}}a_{\beta}(w - z)^{\beta}, \quad w\in\D_{R}(z).
\]
Then we have for every $e'\in G:=E'$ that
\[
 (e'\circ f)(w)=\sum_{\beta\in\N_{0}^{d}}e'(a_{\beta})(w - z)^{\beta}, \quad w\in\D_{R}(z),
\]
implying the holomorphy of $e'\circ f$ by \cite[Theorem 1.7.6, p.\ 48-49]{pflug2008}.\\
`$e)(i)\Rightarrow e)(ii)\Rightarrow e)(iii)\Rightarrow e)(i)$': The first implication is obvious, the second and 
the third follow from the scalar version of Hartogs' theorem (see \cite[Theorem 2.2.8, p.\ 28]{H3}).\\
`$a)(i)\Rightarrow a)(ii)\Rightarrow e)(ii)$' and `$a)(i)\Rightarrow a)(iii)\Rightarrow e)(iii)$': 
These implications are obvious.\\
`$a)(i)\Rightarrow b)\Rightarrow c)$': \prettyref{rem:properties.of.holom.map} b)+c) yields the first implication 
and the second is trivial.\\
`$c) \Rightarrow e)(i)$': This implication holds with $G:=E'$ due to $E'=(\widehat{E})'$ and the scalar version of 
Hartogs' theorem.\\
`$e)(i) \Rightarrow d)$': Let $f\colon\Omega\to E$ be such that $e'\circ f$ is holomorphic on $\Omega$ for every $e'\in G$. 
Then $e'\circ f\circ\phi^{-1}$ is $\mathcal{C}^{\infty}$ on $\phi(\Omega)$ for each $e'\in G$ and we even obtain 
$f\circ\phi^{-1}\in\mathcal{C}^{\infty}(\phi(\Omega),E)$ by the weak-strong principle 
\cite[Corollary 10 (a), p.\ 233]{B/F/J}. Furthermore, the holomorphy of $e'\circ f$
implies
\[
\langle e', (\overline{\partial}_{j})^{E}(f\circ\phi^{-1})(x)\rangle
=(\overline{\partial}_{j})^{\C}\bigr(e'\circ (f\circ\phi^{-1})\bigl)(x)
=0,\quad x\in\phi(\Omega),\,e'\in G,
\]
for every $1\leq j\leq d$ by \cite[Definition 2.1.1, p.\ 23]{H3}. 
We conclude that $d)$ is valid since $G$ is separating.\\ 
`$d) \Rightarrow e)(i)$': Let $f\in\mathcal{C}^{\infty}_{\R}(\Omega,E)$ and $\overline{\partial}_{j}(f\circ\phi^{-1})=0$ 
for all $1\leq j\leq d$. Then $f\circ\phi^{-1}$ is weakly $\mathcal{C}^{\infty}$ on $\phi(\Omega)$ and 
\[
(\overline{\partial}_{j})^{\C}\bigr(e'\circ (f\circ\phi^{-1})\bigl)(x)
=\langle e', (\overline{\partial}_{j})^{E}(f\circ\phi^{-1})(x)\rangle
=0,\quad x\in\phi(\Omega),\,e'\in E',
\]
for every $1\leq j\leq d$. We derive $e)(i)$ with $G:=E'$ from the scalar version of Hartogs' theorem 
and \cite[Definition 2.1.1, p.\ 23]{H3}.\\
`$a)(i) \Rightarrow f)(i)$': A consequence of \prettyref{rem:properties.of.holom.map} b)+d) 
with $K:=\varnothing$ and $G:=E'$.\\
`$f)(i) \Leftrightarrow f)(ii) \Leftrightarrow f)(iii)$': Follows from the correponding equivalences in case e).\\
`$f)(i) \Rightarrow c)$': Fix $1\leq j\leq d$ and $z\in\Omega$ and consider 
the map $f\circ\pi_{j,z}\colon\pi_{j,z}^{-1}(\Omega)\to E$. 
Let $w\in\pi_{j,z}^{-1}(\Omega)\setminus\pi_{j,z}^{-1}(K)$. Then $\pi_{j,z}(w)\in\Omega\setminus K$ and thus 
there is a neighbourhood $U\subset\Omega$ of $\pi_{j,z}(w)$ such that $f$ is bounded on $U$, implying 
that $f\circ\pi_{j,z}$ is bounded on the neighbourhood $\pi_{j,z}^{-1}(U)$ of $w$.  
Thus we can apply \cite[5.2 Theorem, p.\ 35]{grosse-erdmann1992} to $f\circ\pi_{j,z}$ and obtain 
that $(f\circ\pi_{z,j})'(w)$ exists in $E$ for all $w\in\pi_{j,z}^{-1}(\Omega)$, implying for $w=z_{j}$
\[
\partial^{e_{j}}_{\C}f(z)=(f\circ\pi_{z,j})'(z_{j})\in E.
\]
\end{proof}

The preceding theorem generalises corresponding theorems for $E$-valued holomorphic functions in one variable 
given in \cite[Satz 10.11, p.\ 241]{Kaballo} for quasi-complete $E$, 
in \cite[Th\'eor\`eme 1, p.\ 37-38]{Grothendieck1953} (cf.\ \cite[16.7.2 Theorem, p.\ 362-363]{Jarchow}) 
for $E$ with ccp and more general 
in \cite[2.1 Theorem and Definition, p.\ 17-18]{grosse-erdmann1992} and 
\cite[5.2 Theorem, p.\ 35]{grosse-erdmann1992} for locally complete $E$. 
In several variables our theorem improves \cite[Theorem 3.2, p.\ 83-84]{bochnak1971} where $E$ has to be sequentially complete
and even in one variable it is more general than the mentioned ones due to 
\prettyref{thm:holom.equiv} c). The equivalence `$a)(i)\Leftrightarrow a)(iii)$' is Hartogs' theorem 
and can be found for Banach-valued holomorphic functions on an open set $\Omega\subset\C^{d}$ 
in \cite[36.1 Theorem, p.\ 265]{mujica1985} and for holomorphic functions with values 
in a sequentially complete space in \cite[Corollary 3.6, p.\ 85]{bochnak1971}. 
The equivalence `$a)(i)\Leftrightarrow e)(i)$' is also contained in \cite[Corollary 10 (a), p.\ 233]{B/F/J}.

The following two corollaries improve \cite[Corollary 3.7, p.\ 85]{bochnak1971} from sequentially complete $E$ to locally complete $E$.

\begin{cor}\label{cor:compl.deriv.are.holom}
Let $E$ be a locally complete lcHs over $\C$ and $\Omega\subset\C^{d}$ open. 
If $f\colon\Omega\to E$ is holomorphic, then $\partial^{\beta}_{\C}f$ is holomorphic for all $\beta\in\N_{0}^{d}$.
\end{cor}
\begin{proof}
Let $\beta\in\N_{0}^{d}$ and $1\leq j\leq d$. Then we deduce from \prettyref{cor:compl.deriv.commute} 
and Cauchy's integral formula that 
\[
 \partial^{e_{j}}_{\C}\bigl(\partial^{\beta}_{\C}f\bigr)(z)=\partial^{\beta+e_{j}}_{\C}f(z)\in E,\quad z\in\Omega.
\]
It follows from \prettyref{thm:holom.equiv} `$a)\Leftrightarrow c)$' that $\partial^{\beta}_{\C}f$ is holomorphic.
\end{proof}

Defining the subspace $\mathcal{C}^{\infty}_{\C}(\Omega,E)$ of $\mathcal{D}^{\infty}_{\C}(\Omega,E)$ 
which consists of all elements of $\mathcal{D}^{\infty}_{\C}(\Omega,E)$ such that all complex partial derivatives of any order 
are continuous, we state the following consequence of the corollary above.

\begin{cor}\label{cor:hol_infinitely_compl_diff}
Let $E$ be a locally complete lcHs over $\C$ and $\Omega\subset\C^{d}$ open. 
Then the following statements are equivalent for a function $f\colon\Omega\to E$.
\begin{enumerate}
 \item [a)] $f$ is holomorphic on $\Omega$.
 \item [b)] $f\in\mathcal{C}^{\infty}_{\C}(\Omega,E)$.
\end{enumerate}
\end{cor}
\begin{proof}
The implication `$b) \Rightarrow a)$' follows from \prettyref{thm:holom.equiv} `$c)\Rightarrow a)$'. The other implication 
is a consequence of \prettyref{cor:compl.deriv.commute} and \prettyref{cor:compl.deriv.are.holom},
which gives that the holomorphic function $\partial^{\beta}_{\C}f$, $\beta\in\N_{0}^{d}$, is continuous 
by \prettyref{rem:properties.of.holom.map} b).
\end{proof}

\section{Various further results}

The following generalisation of \prettyref{prop:first.real.compl.part.deriv} describes the relation between higher order 
real and complex partial derivatives. For convenience we recall the definition of higher real partial derivatives. 
Let $k\in\N_{0,\infty}$, $\Omega\subset\R^{d}$ open, $E$ an lcHs and $f\in\mathcal{C}^{k}(\Omega,E)$. 
For $\beta=(\beta_{1},\ldots,\beta_{d})\in\N_{0}^{d}$ with 
$|\beta|:=\sum_{j=1}^{d}\beta_{j}\leq k$ we set 
$\partial^{\beta_{j}}_{\R}f:=(\partial^{\beta_{j}}_{\R})^{E}f:=f$ if $\beta_{j}=0$, and
\[
\partial^{\beta_{j}}_{\R}f:=(\partial^{\beta_{j}}_{\R})^{E}f
:=\underbrace{(\partial^{e_{j}}_{\R})\cdots(\partial^{e_{j}}_{\R})}_{\beta_{j}\text{-times}}f
\]
if $\beta_{j}\neq 0$ as well as 
\[
\partial^{\beta}_{\R}f:=(\partial^{\beta}_{\R})^{E}f
:=(\partial^{\beta_{1}}_{\R})\cdots(\partial^{\beta_{d}}_{\R})f.
\]

\begin{prop}\label{prop:real.compl.part.deriv}
If $E$ is a locally complete lcHs over $\C$, $\Omega\subset\C^{d}$ open 
and $f\colon\Omega\to E$ holomorphic, then $f\in\mathcal{C}^{\infty}_{\R}(\Omega,E)$ and 
\begin{equation}\label{eq:real.compl.part.deriv.1}
\partial^{\beta}_{\R}(f\circ \phi^{-1})(x)=i^{\sum_{k=1}^{d}\beta_{2k}}
\partial^{(\beta_{1}+\beta_{2},\ldots,\beta_{2d-1}+\beta_{2d})}_{\C}f(\phi^{-1}(x)),
\quad x\in\phi(\Omega),\,\beta\in\N_{0}^{2d}.
\end{equation}
\end{prop}
\begin{proof}
Due to \prettyref{thm:holom.equiv} `$a)\Leftrightarrow d)$' we have $f\in\mathcal{C}^{\infty}_{\R}(\Omega,E)$, 
implying that the left-hand side of \eqref{eq:real.compl.part.deriv.1} is defined whereas the right-hand side is defined 
by Cauchy's integral formula. Now, for $\beta=0$ equation \eqref{eq:real.compl.part.deriv.1} is trivial. 
Let $n\in\N_{0}$ and assume that \eqref{eq:real.compl.part.deriv.1} holds for all 
$\beta\in\N_{0}^{2d}$ with $|\beta|=n$. 
Let $\beta\in\N_{0}^{2d}$ with $|\beta|=n+1$. Then there is $j\in\N$, 
$1\leq j\leq 2d$, and $\widetilde{\beta}\in\N_{0}^{2d}$ with $|\widetilde{\beta}|=n$ 
such that $\beta=\widetilde{\beta}+e_{j}$. We set 
\[
g:=\partial^{(\widetilde{\beta}_{1}+\widetilde{\beta}_{2},\ldots,\widetilde{\beta}_{2d-1}+\widetilde{\beta}_{2d})}_{\C}f
\quad\text{and}\quad C:=i^{\sum_{k=1}^{d}\widetilde{\beta}_{2k}}
\]
and obtain for $x\in\phi(\Omega)$ by Schwarz' theorem and the induction hypothesis
\[
\partial^{\beta}_{\R}(f\circ \phi^{-1})(x)
=\partial^{e_{j}}_{\R}\bigl[\partial^{\widetilde{\beta}}_{\R}(f\circ \phi^{-1})\bigr](x)
=C\cdot\partial^{e_{j}}_{\R}(g\circ\phi^{-1})(x).
\]
We deduce from \prettyref{cor:compl.deriv.are.holom} that $g$ is holomorphic and from \prettyref{cor:hol.in.O} 
that $g\in\mathcal{O}(\Omega,E)$.
Thus we can apply \prettyref{prop:first.real.compl.part.deriv} to $g$ and compute
\[
\partial^{e_{j}}_{\R}(g\circ\phi^{-1})(x)
=i\partial^{e_{j}}_{\C}g(\phi^{-1}(x)),
\]
if $j$ is even, and
\[
\partial^{e_{j}}_{\R}(g\circ\phi^{-1})(x)
=\partial^{e_{j}}_{\C}g(\phi^{-1}(x)),
\]
if $j$ is odd, which implies by \prettyref{cor:compl.deriv.commute} 
\[
\partial^{\beta}_{\R}(f\circ \phi^{-1})(x)
=C\cdot\partial^{e_{j}}_{\R}(g\circ\phi^{-1})(x)
=i^{\sum_{k=1}^{d}\beta_{2k}}\partial^{\beta}_{\C}f(\phi^{-1}(x)).
\]
\end{proof}

We call a non-empty subset $A$ of an lcHs $E$ \emph{locally closed} if every local limit point of $A$ belongs to $A$. A point $x\in E$ is called a local limit point of $A$ if there is a sequence $(x_{n})_{n\in\N}$ in $A$ that converges locally to $x$ (see \cite[Definition 5.1.14, p.\ 154-155]{Bonet}), i.e.\ there is a disk $D\subset E$ such that $(x_{n})$ converges to $x$ in $E_{D}$ (see \cite[Definition 5.1.1, p.\ 151]{Bonet}). 
In particular, every locally complete 
subspace of $E$ is locally closed by \cite[Proposition 5.1.20 (i), p.\ 155]{Bonet}.
Now, the identity theorem for vector-valued holomorphic functions in several variables takes the following form 
where the vector-valued one variable case can be found in \cite[Corollary 10 (c), p.\ 233]{B/F/J} and 
the scalar-valued several variables case for example in \cite[Proposition 1.7.10, p.\ 50]{pflug2008}. 
Its version for Banach-valued holomorphic functions on an open subset of a Banach space is given in 
\cite[5.7 Proposition, p.\ 37]{mujica1985}.

\begin{thm}[{Identity theorem}]\label{thm:identity}
Let $E$ be a locally complete lcHs over $\C$, $F\subset E$ a locally closed subspace, 
$\Omega\subset\C^{d}$ open and connected and $f\colon\Omega\to E$ holomorphic. 
If 
\begin{enumerate}
 \item[(i)] the set $\Omega_{F}:=\{z\in\Omega\;|\;f(z)\in F\}$ has an accumulation point in $\Omega$, or if
 \item[(ii)] there exists $z_{0}\in\Omega$ such that $\partial^{\beta}_{\C}f(z_{0})\in F$ for all $\beta\in\N_{0}^{d}$,
\end{enumerate}
then $f(z)\in F$ for all $z\in\Omega$.
\end{thm}
\begin{proof}
This follows from \prettyref{prop:eps-prod-holom} and \cite[Corollary 8, p.\ 232]{B/F/J} 
with the Fr\'echet-Schwartz space $Y:=\mathcal{O}(\Omega)$ and the separating subspace 
$X:=\operatorname{span}\{\delta_{z}\;|\;z\in\Omega_{F}\}\subset Y'$ in (i) 
resp.\ $X:=\operatorname{span}\{\delta_{z_{0}}\circ\partial^{\beta}_{\C} \;|\;\beta\in\N_{0}^{d}\}\subset Y'$ in (ii).
\end{proof}

\begin{thm}[{Liouville}]\label{thm:liouville}
Let $E$ be a locally complete lcHs over $\C$, $f\colon\C^d\to E$ holomorphic and $k\in\N_{0}$. 
Then the following assertions are equivalent.
\begin{enumerate}
 \item[a)] $f$ is a polynomial of degree $\leq k$.
 \item[b)] $\forall\alpha\in\mathfrak{A}\;\exists\;C,R>0\;\forall\;z\in\C^{d},\,|z|\geq R:\;p_{\alpha}(f(z))\leq C|z|^k$
\end{enumerate}
\end{thm}
\begin{proof}
 The implication `$a)\Rightarrow b)$' is obvious and the converse implication holds due to the power series expansion 
 \prettyref{thm:powerseries} of $f$ around zero and the Cauchy inequality \eqref{eq:cauchy.ineq.1}. 
\end{proof}

Let $\Omega\subset\C^{d}$ be open and connected. A set $A\subset\Omega$ is called \emph{thin} 
if for every $z\in A$ there 
are $R>0$ with $\D_{R}(z)\subset\Omega$ and $f\in\mathcal{O}(\D_{R}(z))$, $f\neq 0$, such that that $f=0$ on $A\cap\D_{R}(z)$ 
(see e.g.\ \cite[Chap.\ 1, Sec.\ C, 1.\ Definition, p.\ 19]{gunning1965}). 
A thin set $A\subset\Omega$ is \emph{nowhere dense} by \cite[p.\ 19]{gunning1965} and thus the complement $\Omega\setminus A$ contains 
a dense open subset.

\begin{thm}[{Riemann's removable singularities theorem}]\label{thm:riemann}
Let $E$ be an lcHs over $\C$, $G\subset E'$ a subspace, $\Omega\subset\C^{d}$ open and connected, 
$A\subset\Omega$ thin and closed and $f\colon\Omega\setminus A\to E$ such that $e'\circ f$ is holomorphic for each $e'\in G$. 
If for every $z\in\Omega$ there is a polydisc $\D_{R}(z)\subset\Omega$ such that $f$ is bounded on 
$\D_{R}(z)\setminus A$ and
\begin{enumerate}
 \item[(i)] $G$ is separating and $E$ $B_{r}$-complete, or if
 \item[(ii)] $G$ is dense in $E_{b}'$ and $E$ locally complete,
\end{enumerate}
then $f$ extends holomorphically to $\Omega$.
\end{thm}
\begin{proof}
First, we remark that $e'\circ f$ is holomorphic and bounded on $\D_{R}(z)\setminus A$ 
for some $R$ and each $z\in\Omega$ and $e'\in G$. 
Due to the scalar version of Riemann's removable singularities theorem (see \cite[Chap.\ 1, Sec.\ C, 3.\ Theorem, p.\ 19]{gunning1965}) 
$e'\circ f$ extends to a holomorphic function $f_{e'}$ on $\Omega$ for each $e'\in G$. 
Let $(\Omega_{n})_{n\in\N}$ be any exhaustion of $\Omega$ with relatively compact, open and connected sets 
such that $\Omega_{n}\subset\Omega_{n+1}$ for every $n\in\N$. Since $M:=\Omega\setminus A$ is dense in $\Omega$, we have 
$\partial\Omega_{n}\subset\overline{\Omega}_{n+1}=\overline{M\cap\Omega_{n+1}}$. Hence our statement is true by 
\cite[Corollary 18, p.\ 238]{B/F/J} with $\mathscr{F}(\Omega):=\mathcal{O}(\Omega)$.
\end{proof}

\emph{$B_{r}$-complete spaces} (see \cite[p.\ 183]{Jarchow}) are also called \emph{infra-Pt\'ak spaces} and, 
for instance, every Fr\'echet space is $B_{r}$-complete 
by \cite[9.5.2 Krein-\u{S}mulian Theorem, p.\ 184]{Jarchow}. Condition (ii) is especially fulfilled 
if $G$ is separating and $E$ semireflexive (see \cite[p.\ 234]{B/F/J}). 
Choosing $G=E'$, we get the more familiar version of Riemann's removable singularities theorem from 
\prettyref{thm:riemann} (ii) if $E$ is locally complete.

\begin{cor}
Let $E$ be a locally complete lcHs over $\C$, $\Omega\subset\C^{d}$ open and connected, 
$A\subset\Omega$ thin and closed and $f\colon\Omega\setminus A\to E$ holomorphic. 
If for every $z\in\Omega$ there is a polydisc $\D_{R}(z)\subset\Omega$ such that $f$ is bounded on 
$\D_{R}(z)\setminus A$, then $f$ extends holomorphically to $\Omega$.
\end{cor}

For holomorphic functions in infinitely many variables with values in a sequentially complete space $E$ 
the corresponding result is given in \cite[Corollary 5.2, p.\ 95]{bochnak1971}.

We close our treatment of holomorphic functions with another application 
of the power series expansion given in \prettyref{thm:powerseries}.
For an lcHs $E$ over $\C$, $z\in\C^{d}$ and $R\in(0,\infty)^{d}$ we set
\[
 \mathcal{A}(\overline{\D_{R}(z)},E)
 :=\mathcal{O}(\D_{R}(z),E)\cap\mathcal{C}^{0}(\overline{\D_{R}(z)},E)
\]
and equip this space with the system of seminorms generated by
\[
\|f\|_{\alpha}:=\sup_{w\in\overline{\D_{R}(z)}}p_{\alpha}(f(w)),\quad f\in \mathcal{A}(\overline{\D_{R}(z)},E),
\]
for $\alpha\in\mathfrak{A}$. 
We write $\mathcal{A}(\overline{\D_{R}(z)}):=\mathcal{A}(\overline{\D_{R}(z)},\C)$ and in the case $z=0$ and 
$R=(1,\ldots,1)$ this space is known as the \emph{polydisc algebra}.
Further, we denote by $\mathcal{P}(\overline{\D_{R}(z)},E)$ the space of $E$-valued polynomials on 
$\overline{\D_{R}(z)}$ and aim to prove that the $E$-valued polynomials are dense 
in $\mathcal{A}(\overline{\D_{R}(z)},E)$ if $E$ is locally complete. 
If $E$ is a Fr\'echet space, this result can be found in \cite[2.\ Theorem (2), p.\ 2]{bierstedt1969}. 
Our proof is along the lines of the one for $E=\C$, $z=0$, $d=1$ and $R=1$ given in 
\cite[p.\ 366]{rudin1986}.

\begin{cor}\label{cor:disc-algebra}
 Let $E$ be a locally complete lcHs over $\C$, $z\in\C^{d}$ and $R\in(0,\infty)^{d}$. 
 Then the following statements hold. 
 \begin{enumerate}
 \item [a)] The tensor product $\mathcal{P}(\overline{\D_{R}(z)},\C)\otimes E$ is dense in 
 $\mathcal{A}(\overline{\D_{R}(z)},E)$.
 \item [b)] If $E$ is complete, then 
 \[
  \mathcal{A}(\overline{\D_{R}(z)},E)\cong \mathcal{A}(\overline{\D_{R}(z)})\varepsilon E
  \cong \mathcal{A}(\overline{\D_{R}(z)})\widehat{\otimes}_{\varepsilon} E
 \]
 where $\cong$ stands for topologically isomorphic and 
 $\mathcal{A}(\overline{\D_{R}(z)})\widehat{\otimes}_{\varepsilon} E$ 
 is the completion of the injective tensor product $\mathcal{A}(\overline{\D_{R}(z)})\otimes_{\varepsilon} E$.
 \item [c)] $\mathcal{A}(\overline{\D_{R}(z)})$ has the approximation property.
 \end{enumerate}
\end{cor}
\begin{proof}
The map 
\[
S\colon\mathcal{A}(\overline{\D_{R}(z)})\varepsilon E\to\mathcal{A}(\overline{\D_{R}(z)},E),\;
u\longmapsto [w\mapsto u(\delta_{w})],
\]
is a (topological) isomorphism into, i.e.\ to its range, by \cite[3.1 Bemerkung, p.\ 141]{B2} 
with $Y:=\mathcal{A}(\overline{\D_{R}(z)})$ 
and \prettyref{thm:holom.equiv} `$a)\Leftrightarrow e)$'. 
Let $\alpha\in\mathfrak{A}$, $\varepsilon>0$ and $f\in\mathcal{A}(\overline{\D_{R}(z)},E)$. 
Since $\overline{\D_{R}(z)}$ is compact, $f$ is uniformly continuous on $\overline{\D_{R}(z)}$ 
and thus there is $\delta>0$ such that $p_{\alpha}(f(w)-f(x))<\varepsilon$ for all $x,w\in\overline{\D_{R}(z)}$ 
with $|w-x|<\delta$. Choosing $r>0$ such that 
\[
 1-\frac{\delta}{\sqrt{d}\max_{1\leq j\leq d} (R_{j}+|z_{j}|)}<r<1,
\]
we get $1/r>1$ and thus $\overline{\D_{R}(z)}\subset\D_{R/r}(z)$. Further, for 
every $w\in\overline{\D_{R}(z)}$ we have
\[
 |w|\leq \sqrt{d}\max_{1\leq j\leq d}|w_{j}|\leq \sqrt{d}\max_{1\leq j\leq d}(|w_{j}-z_{j}|+|z_{j}|)
 \leq \sqrt{d}\max_{1\leq j\leq d} (R_{j}+|z_{j}|),
\]
implying
\[
 |w-rw|=(1-r)|w|<\frac{\delta}{\sqrt{d}\max_{1\leq j\leq d} (R_{j}+|z_{j}|)}|w|\leq \delta.
\]
Therefore
\begin{equation}\label{eq:disc-algebra.1}
 \sup_{w\in\overline{\D_{R}(z)}}p_{\alpha}(f(w)-f(rw))\leq\varepsilon
\end{equation}
and we set $g\colon \D_{R/r}(z)\to E,\;g(w):=f(rw)$, which is a function in $\mathcal{O}(\D_{R/r}(z),E)$. 
By \prettyref{thm:powerseries} there is $N\in\N$ such that for all $n\geq N$ 
\[
 \bigl\|g-\underbrace{\sum_{|\beta|\leq n}(\cdot- z)^{\beta}\otimes 
 \frac{(\partial^{\beta}_{\C})^{E}g(z)}{\beta!}}_{=:T_{n}}\bigr\|_{\alpha}
=\sup_{w\in\overline{\D_{R}(z)}}p_{\alpha}\bigl(g(w)-
 \sum_{|\beta|\leq n}\frac{(\partial^{\beta}_{\C})^{E}g(z)}{\beta!}(w - z)^{\beta}\bigr)
<\varepsilon.
\]
We observe that the restriction of $T_{N}$ to $\overline{\D_{R}(z)}$ is an element
of $\mathcal{P}(\overline{\D_{R}(z)},\C)\otimes E=\mathcal{P}(\overline{\D_{R}(z)},E)$ 
and thus of $\mathcal{A}(\overline{\D_{R}(z)})\otimes E$ 
as well. We conclude from \eqref{eq:disc-algebra.1} that
\[
 \|f-T_{N}\|_{\alpha}\leq \|f-g\|_{\alpha}+\|g-T_{N}\|_{\alpha}<2\varepsilon,
\]
which proves our first statement. The second follows from the first by \cite[3.5 Remark, p.\ 7]{kruse2018_1} 
because $\mathcal{A}(\overline{\D_{R}(z)})$ is a Banach space and thus complete. 
The last statement results from the second by \cite[18.1.8 Theorem, p.\ 400]{Jarchow}.
\end{proof}

The consequences b) and c) of a) are also to be found in \cite[p.\ 5]{bierstedt1969} and 
\cite[Theorem 3.5, p.\ 561]{eifler1969} in combination with Mergelyan's theorem, respectively.

\subsection*{Acknowledgement}
We are thankful to the anonymous reviewer for the thorough 
review and helpful suggestions.

\bibliography{biblio}
\bibliographystyle{plain}
\end{document}